\documentclass[12pt,reqno]{amsart}
\newcommand{\ver}{{\it 1c}}

\usepackage{latexsym}
\newcommand{\X}{{\mathbf X_{\Gamma}}}

\newcommand{\R}{{\mathbb R}}
\newcommand{\C}{{\mathbb C}}

\newcommand{\Z}{{\mathbb Z}}

\renewcommand{\H}{{\mathbf H}}

\newcommand{\half}{{\frac{1}{2}}}

\renewcommand{\phi}{\varphi}

\renewcommand{\epsilon}{\varepsilon}

\newcommand{\kahler}{K\"ahler }

\renewcommand{\phi}{\varphi}

\newcommand{\ccal}{\mathcal{C}}

\newcommand{\gcal}{\mathcal{G}}

\newcommand{\mcal}{\mathcal{M}}
\newcommand{\ocal}{\mathcal{O}}

\newcommand{\rcal}{\mathcal{R}}
\newcommand{\scal}{\mathcal{S}}

\usepackage[active]{srcltx}

\newtheorem{theo}{{\sc Theorem}}[section]
\newtheorem{cor}[theo]{{\sc Corollary}}

\newtheorem{lem}[theo]{{\sc Lemma}}
\newtheorem{prop}[theo]{{\sc Proposition}}

\newenvironment{rem}{\medskip\noindent{\it Remark:\/} }{\medskip}
\newenvironment{defin}{\medskip\noindent{\it Definition:\/} }{\medskip}

\newtheorem{mainprop}{{\sc Proposition}}

\newtheorem{maincor}{{\sc Corollary}}

\title[Matrix elements of Fourier integral operators\ver] {Matrix elements of Fourier Integral Operators\\
}

\author{Steve Zelditch}
\address{Department of Mathematics, Northwestern  University,
Evanston, IL 60208-2370, USA} \email{
zelditch@math.northwestern.edu}

\thanks{Research partially supported by NSF grant  \# DMS-1206527.}

\begin{document}


 \maketitle

\begin{abstract}  This article is concerned with the semi-classical limits of matrix
elements $\langle F \phi_j, \phi_j \rangle$ of eigenfunctions of the Laplacian $\Delta_g$
of a compact Riemannian manifold $(M, g)$  with respect to a Fourier
integral operator $F$ on $L^2(M)$. More generally, we consider matrix elements of eigensections of quantum maps. 
Many results exist for the case where $F$ is a pseudo-differential operator, but matrix elements of
Fourier integral operators involve new considerations. The limits reflect the extent to which
the canonical relation of $F$ is invariant under the geodesic flow of $(M, g)$. When the canonical relation
is almost nowhere invariant, a density one subsequence of the  matrix elements tends to zero (related results
arose first in the study of quantum ergodic restriction theorems).  The limit states are invariant measures
on the canonical relation of $F$ and their invariance properties are explained. The invariance properties
in the case of Hecke operators answers an old question raised by the author in \cite{Z1}.

\end{abstract}

One of the main objects of study  in quantum ergodicity  is  the
sequence  of diagonal  matrix elements \begin{equation} 
\label{rho} \rho_{j} (A) = \langle
 A
 \phi_j, \phi_j \rangle \end{equation} of  zeroth order pseudodifferential
 operators $A \in \Psi^0(M)$  relative to an orthonormal basis $\{\phi_j\}$ of eigenfunctions
 $$\Delta \phi_j = \lambda_j^2 \; \phi_j,\;\;\; \langle \phi_j,
 \phi_k \rangle = 0. $$
of the Laplacian $\Delta$ of a compact Riemannian manifold $(M,
 g)$.
The diagonal matrix element \eqref{rho} define positive linear functionals  of mass one,
\begin{equation} \rho_j: \overline{\Psi}^0 \to \R, \;\;\;\; \rho(I) = 1, \end{equation}
on the norm closure of the space  $ \Psi^0(M)$ of zeroth order pseudo-differential operators, and  are invariant
under the wave group in the sense that
\begin{equation} \label{WG}  \rho_j (U^{t *} A U^t) = \rho_j. \end{equation}
The well-known consequence of Egorov's theorem is  that any weak* limit  $\mu$ of the sequence $\{\rho_j\}$ lies in the space
$\mcal_I$ of
 invariant probability measure for the geodesic flow $G^t$ on
 $S^*M$, i.e. is a positive linear functional on $C(S^*M)$ with $\mu(1) = 1$ and $G^t_* \mu = \mu$. Moreover,
one has the local Weyl law
\begin{equation} \label{WLW} \lim_{\lambda \to \infty} \frac{1}{N(\lambda)} \sum_{j: \lambda_j \leq \lambda} \rho_j = \omega_L, \end{equation}
where $\omega_L(A) = \int_{S^* M} \sigma_A d\mu_L$ is the Liouville state of integration of the  principal
symbol of $A$ with respect to normalized Liouville measure. Also, $N(\lambda) = \# \{j : \lambda_j \leq \lambda\}$
is the Weyl counting function and convergence is in the sense of continuous linear functionals on $\Psi_0$. We
refer to \cite{Z3,Z4, Zw} for background on these statements. 
The off-diagonal matrix elements are also important and we refer to \cite{Z5} for results on them.

The purpose of the present note is to consider the analogues of these basic 
results for diagonal  (and to a lesser extent,  off-diagonal)  matrix elements \eqref{rho} of   Fourier integral operators $F$ associated to a closed
canonical relation  $C \subset T^* M \times T^* M$. That is, we consider the functionals
\begin{equation} \label{rhojF} \rho_j: I^r(M \times M, C)  \to \C, \;\;\; \rho_j(F) = \langle F \phi_j, \phi_j \rangle,
\end{equation}
where $ I^r(M \times M, C) $ is the space of Fourier integral operators of order $r$  with wave front
relation along $C$ (see Vol. 4 of \cite{HoI-IV}   for background and notation).  
The linear
functionals  $\rho_j$ \eqref{rhojF}  are   invariant under the wave group in the
sense that if $U^t = e^{i t \sqrt{\Delta}}$ is the wave group of
$(M, g)$, then
\begin{equation} \label{QINV} \rho_j^t(F) : = \rho_j(U^{- t} F U^t) = \rho_j(F).  \end{equation}   In particular, we are interested in the weak*
limits of the sequence $\{\rho_j\}$. 

\begin{defin} We define a weak* limit  $\rho_{\infty} $ of the functionals $\{\rho_j\}$
to be a functional on $I^0(M \times M, C)$ so that
$\rho_j(F) \to \rho_{\infty}(F)$ for all $F$ in this class. 

 \end{defin} 
It is shown in Proposition \ref{w*L} that $\rho_{\infty}(F)$ depends only on the principal symbol of $F$, i.e.
on the half-density symbol on the associated canonical relation and defines a measure on the symbols. We
therefore use a convenient abuse of notation and identity the state and the measure, i.e. we put
\begin{equation} \label{rhoinf} \rho_{\infty}(F) = \rho_{\infty}(\sigma_F). \end{equation} 

 The motivation for this problem comes from several sources, for instance:

\begin{itemize}

\item (i)  The question of weak * limits for matrix elements of Fourier integral operators arose in recent work on quantum ergodic restriction theorems
\cite{TZ,TZ2,DZ}, and more recently for ray-splitting in \cite{JSS}.  For the $F$ in those articles, the underlying canonical relation is a local canonical graph. 
 A key point was that $\rho_j(F) \to 0$ along a subsequence
of density one when the local canonical graphs are  `almost nowhere invariant' under the geodesic
flow. 
 We say that $F \phi_j$ is almost orthogonal to $\phi_j$. One aim of this note is to understand such almost
orthogonality  more systematically.

\item  (ii) Hecke operators $T_p$
are Fourier integral operators associated to local canonical graphs, namely $C$ is the lift to $G/\Gamma$ of the Hecke correspondence \cite{E}.   All work in 
arithmetic quantum chaos concerns joint eigenfunctions of $\Delta$ and of the Hecke operators $T_p$. 
An obvious question is the Hecke correspondence  invariance properties of the weak * limits of Hecke $\rho_j$.  This
question was raised but not settled in \cite{Z1} and an answer will be given in Proposition \ref{HECKESN}
for Hecke operators on spheres and in Proposition \ref{HECKEHYP} for arithmetic hyperbolic quotients . Of course, Lindenstrauss
\cite{L} has long since proved that $\rho_j \to \omega_L$,  but the result we present appears to be new.

\item (iii)  A  weak* limit problem  where the canonical relation is not a local canonical graph arose in \cite{Z10}. 
To study nodal sets, the $\Delta$-eigenfunctions    of a real analytic Riemannian
manifold $(M, g)$ were analytically continued $\phi_j \to \phi_j^{\C}$  to its Grauert tubes $M_{\tau}$ and then restricted to geodesic arcs $\gamma: \R \to \partial M_{\epsilon}$. The pullbacks $\gamma^* |\phi_j^{\C}|^2$ can be normalized
to form a bounded sequence of measures on compact intervals of $\R$. All of their weak * limits are constant multiples
of Lebesgue measure. The constants  depend on whether the geodesic is closed or not. Related pointwise  Weyl laws for 
$|\phi_j^{\C}(\zeta)|^2$ on all of $\partial M_{\epsilon}$  have  been proved \cite{Z8}.

\item (iv)  Pointwise phase space Weyl laws for matrix elements of coherent state projectors
$F_{\hbar} = \psi_{x, \xi}^{\hbar_j} \otimes \psi_{x, \xi}^{\hbar_j *}$ 
were obtained in   \cite{PU}. They are somewhat similar to modulus squares 
$|\phi_j^{\C}(\zeta)|^2$ but involve a different FBI transform.   It does not seem that the weak* limit problem was studied 
explicitly before, but the results are rather similar to the restrictions  $\gamma^* |\phi_j^{\C}|^2$. 

\item (v) Both of the above problems are special cases of weak* limit problems for $\rho_j$ on algebras of
Toeplitz operators associated to invariant symplectic cones $\Sigma$  under the geodesic flow. The results
in this setting are parallel to  the case of $\rho_j$ as states
on the algebra $\Psi^0(M)$ in the sense of \cite{Z4}. Different algebras of Fourier integral operators
associated to idempotent canonical relations were introduced in \cite{GuSt}. In a special case, the the weak* limit problem was
studied in \cite{Z9}  (see also \cite{ST}).


\end{itemize}

There is a long-standing question as to the uniquess of weak* limits \eqref{1} when the geodesic flow is
sufficiently chaotic. The larger the class of  `test' operators one can use, the more control one has over the limits.  In general one would
like to study the most general possible microlocal defect measures. We refer to  \cite{Zw} for general background.

In this article, we concentrate on the case where  $C$ is a local canonical graph and the order  $r = 0$, and
only briefly summarize results on weak * limits of matrix elements in the other cases above. 
In the canonical graph  case,  the  family  $\{\rho_j\}$ 
of functionals on $I^0(M \times M, C)$ is uniformly bounded and all  weak* limits are complex measures
on $C$.   More precisely, they are linear functionals of  the symbol $\sigma_F$
of $F$, which is a $\half$-density 
along $C$ (times a Maslov factor, which will be ignored here for simplicity of exposition and because
the results do not depend on the Maslov factor).  The local Weyl law for  Fourier integral operators 
associated to local canonical graphs was studied  in \cite{Z2} (see also \cite{TZ2, JSS}).

Many of the results for local canonical graphs turn out to be  negative: the weak * limits of the diagonal
functionals $\rho_j$  are very often zero, since most canonical graphs $C$ are almost nowhere invariant
under a given geodesic flow.  The graph in the Hecke case is invariant \S \ref{HECKE}, and so the
question becomes one of determining when the limits are trivial and when they are not. The local canonical
(or isotropic) relation in the case of $\gamma^* |\phi_j^{\C}(z)|^2$ or for $|\langle \phi_j, \psi_{x, \xi}^{\hbar}
\rangle|^2$ are not invariant but can be time averaged to become invariant, and the weak * limits (when
suitably normalized) are often non-trivial. 

The results
also suggest that off-diagonal elements
\begin{equation} \label{rhoij} \rho_{ij}(F)  = \langle F \phi_i, \phi_j \rangle \end{equation}  are often
more natural when testing against  Fourier integral operators. They satisfy
\begin{equation} \label{QINVoff} \rho_{ij}^t(F) : = \rho_{ij} (U^{- t} F U^t) = e^{i t (\lambda_i - \lambda_j)}  \rho_j(F), \end{equation} 
and intuitively correspond to canonical transformations which change the energy level. However, the weak*
limits are again  trivial when the graph is almost nowhere invariant.



\subsection{Results for local canonical graphs}

To state our results, we need to introduce some further notation. A Fourier integral operator is an operator $F$  whose
 Schwartz kernel  may be locally represented as a finite
sum of oscillatory integrals,
\begin{equation} \label{K} K_F(x,y) \sim \int_{\R^N} e^{i \phi(x, y, \theta)} a(x, y,
\theta) d\theta \end{equation}  for some homogeneous phase $\phi$ and amplitude
$a$. It is well known that $F$ is determined up to compact
operators by the  canonical relation
$$C = \{(x, \phi'_x, y, - \phi'_y):  \phi'_{\xi}(x, y, \xi) =
0\} \subset T^*M \times T^*M, $$ together with  the principal
$\sigma_F$ of $F$, a $1/2$-density along $C$.We denote by $I^0(M, \times
M, C)$ the class of Fourier integral operators of order zero and
canonical relation $C$.  We refer to \cite{HoI-IV} for the background. Then we may regard (\ref{rho}) as defining
continuous linear functionals on $I^0(M \times M, C)$ with respect
to the operator norm. We  recall that $C$ is a local canonical graph when both projections in the diagram \begin{equation}\label{DIAGRAMintro} \begin{array}{ccccc} & & C \subset T^*M \times T^* M & &  \\ & & & & \\
& \pi_X \swarrow & & \searrow \pi_Y & \\ & & & & \\
T^* M & & \iff & & T^* M
\end{array}
\end{equation}
are  (possibly branched) covering maps. 
 If we equip $C$ with the symplectic volume measure pulled back by $\pi_X$ from $T^* M$, then we may
consider symbols $\sigma_F$ as functions on $C$. Some well-known examples are:

\begin{itemize}

\item $F = T_g$ is translation by an isometry of a Riemannian
manifold $(M, g)$ possessing an isometry.

\item $F$ is a Hecke operator $T f(x) = \sum_{j = 1}^k  (f(g_j x) +
f(g_j^{-1} x) )$ on $S^n$ or on an arithmetic hyperbolic manifold corresponding to a finite set $\{g_1,
\dots, g_k\}$ of isometries of the universal cover. In this case $C$ is the graph of the cotangent lift of the Hecke
correspondence and is a local canonical graph  \cite{RS,LPS}.

\item $F_t = U^t = e^{i t \sqrt{\Delta}}$ or its self-adjoint part $ \cos t \sqrt{\Delta}$.

\item $F = W^* W$ where $W f = \gamma_H B U^t$ where $\gamma_H$ is restriction to a hypersurface
$H \subset M$ and $WF'(B)$ is disjoint from the cotangent directions to $H$ \cite{Ta, GS, TZ,TZ2,DZ} among many
articles.

\item $F$ is a semi-classical quantum map in the setting of positive Hermitian holomorphic line 
bundles over K\"ahler manifolds \cite{Z6}.

\end{itemize}





The first result is:

\begin{mainprop} \label{w*L}  If $C$ is a local canonical graph and $F \in I^0(M \times M, C)$ then the weak limits $\rho_{\infty}$ of $\langle F \phi_j, \phi_j \rangle$ are measures on $SC := C \cap S^* M \times S^* M$, i.e. 
$|\rho_{\infty}(F)| \leq C  \sup_{S C}||\sigma_F||_{C^0} $. 
\end{mainprop}
As discussed above \eqref{rhoinf} we also write the limit functional as functional on the symbol, and
thus have $|\rho_{\infty}(\sigma_F)| \leq C  \sup_{S C}||\sigma_F||_{C^0} $. 

The proof is quite similar to that for $A \in \Psi^0(M)$. But it can be useful to interpret the quantum limits
as living on $SC$ rather than on $S^* M$, as will be seen in the case of Hecke operators.

The weak* limit problem for $\{\rho_j\}$ on $I^0(M \times M, C)$ is not so different from that of $\Psi^0(M)$ since
there often exists an elliptic element $F_0$ of $I^0(M \times M, C)$, i.e. one with nowhere vanishing symbol, and then
all of the elements have the form $A F_0$ or $F_0 B$ where $A, B \in \Psi^0(M)$. 
But the canonical relation
of $U^{- t} F U^t$  equals the image
$$C_t := (G^{-t} \times G^t)  (C) $$ of $C$ under the map $G^{-t} \times G^t$ of $T^* M \times T^* M$, and only coincides with the canonical relation $C$
of $F$ if $C$ is $G^t$-invariant.  In general, 
$$U^{- t} F U^t \in I^0(M \times M, C_t), $$
so that by \eqref{QINV} $\rho_j^t$  induces  a functional on $I^0(M \times M, C_t)$.

The following initial result shows that the invariance properties
of quantum limits depend  on whether the canonical
relation is invariant under the geodesic flow. 

\begin{mainprop} \label{1} Let $\rho_{\infty}^t$ be a weak* limit of the functionals $\rho_j^t$
on $I^0(M \times M, C)$. Then there exists a family of measures  $\mu_t$ on $C_t$  such that
$$\rho^t_{\infty} (A) =  \int_{C_t} (G^{t}  \times G^{-t})^* \sigma_A  \;\;\;d\mu_t,  \;\;\;\; A \in I^0(M \times M, C)$$
with
$\mu_t =   \mu$ on $C_t \cap C$.

\end{mainprop}

Thus,

\begin{maincor} \label{2}Let $\rho_{\infty}$ be a limit of the sequence of functionals $\rho_j(F) = \langle F \phi_j, \phi_j \rangle$
on  $I^0(M, \times M, C)$. Suppose that the canonical relation $C$
is invariant under the geodesic flow $G^{-t} \times G^t$. Then  $\rho_{\infty}$ is a  $G^t$-invariant signed  measure of mass
$\leq 1$ on $C$.
\end{maincor}



Of course, the quantum invariance \eqref{QINV} implies that
$$  \int_{C_t} (G^{t}  \times G^{-t})^* \sigma_A  \;\;\;d\mu_t =   \int_{C}  \sigma_A  \;\;\;d\mu. $$
But Proposition \ref{1}  does not give  any non-trivial invariance conditions on    the set where canonical relation $C$
is nowhere invariant under the geodesic flow $G^{-t} \times G^t$. 
The next result identifies the limit measure as zero along a subsequence of density one. 
We refer to this as the  `almost-orthogonality' of $F \phi_j$ and $\phi_j$.  In the following,
 $n = \dim M$ so that $ 2n = \dim C$.

\begin{mainprop} \label{ANCsympcor} Let $F \in I^0(M \times M, C)$ and  assume that for $t > 0$, the
set $\ccal_{t} = C_t \cap C$ has
Minkowski  $2n$- measure zero.  Then there is a density one subsequence of
eigenfunctions so that $\langle F \phi_{\lambda_j},
\phi_{\lambda_j} \rangle \to 0$. \end{mainprop}


A special case of this almost orthogonality result was one of the the main ingredients in the proof in \cite{TZ,TZ2,DZ} of the quantum
ergodic restriction theorem along hypersurfaces. See Theorem 10 of \cite{TZ} or \S 8 of \cite{TZ2}.
Although this statement is reminiscent of quantum ergodicity,  it does not use 
any dynamical properties such as  ergodicity of $G^t$. As Proposition \ref{ANCsympcor}  indicates,
 almost-orthogonality can sometimes 
be understood in terms of localization on energy surfaces of eigenfunctions.  The nowhere commuting condition in that
result
implies that $F \phi_{\lambda}$ localizes on a disjoint set from $\phi_j$  and
thus the two states are almost orthogonal. In this example there are no sparse exceptional
subsequences of eigenfunctions.   See also Lemma \ref{chigtNC} of
\S \ref{Almostdisjoint}. But this is not always the case, for instance localization does not seem to play a 
role in the  K\"ahler analogue of Theorem \ref{KAHLER}.

\subsection{Reality} 

Additional invariance properties arise when $F$ is self-adjoint or real due to the fact that the eigenfunctions
are real valued. 
We say that $F$ is real if $c F = F c$ where $c$ denotes complex conjugation.  We note that
$$\langle F^* \phi_j, \phi_j \rangle = \overline{\langle F \phi_j,
\phi_j \rangle} = \langle \overline{F}^* \phi_j, \phi_j \rangle =
\langle F^t \phi_j, \phi_j \rangle.
$$

We recall that
the transpose of a Lagrangian manifold $\Lambda$ is defined by
$$\Lambda^t = \{(y, \eta, x, \xi): (x, \xi, y, \eta) \in \Lambda\},
$$ i.e. it is the image of $\Lambda$ under the involution,
$$\iota(x, \xi, y, \eta) = (y, \eta, x, \xi). $$
The `conjugate' Lagrangian is defined by
$$\Lambda^* =  \{(y, \eta, x, - \xi): (x, \xi, y, \eta) \in \Lambda\},
$$
i.e. it is the image under $c \circ \iota$ where $c$ is the
conjugation involution $(x, \xi) \to (x, - \xi)$ in the second
variable; the notation is consistent
 since complex conjugation is the quantization of the map $c$. We say that a canonical relation $C$ is symmetric if $C^t
= C$ and self-adjoint if $C^* = C$. 
A self-adjoint Fourier integral operator always has a self-adjoint
canonical relation and a real Fourier integral operator has a
 canonical relation invariant under $c$.  
When $F$ is self-adjoint we obtain an addtional  invariance principle if we consider symbols defined by 
the functions $\pi_X^* a, \pi_Y^* a$ (cf. \eqref{DIAGRAMintro}):
\begin{mainprop}\label{SA}  If $F$ is real and self-adjoint  then 
$(\pi_X)_* \mu  = (\pi_Y)_* \mu$ for any quantum limit measure $\mu$ on $C$  of Proposition \ref{1}.

\end{mainprop}

This additional principle is useful in obtaining relations between limit measures of Hecke operators.
See \S \ref{HECKE}.

\subsection{Quantum maps in the K\"ahler setting}

There is a natural analogue of Proposition \ref{ANCsympcor} in the  \kahler setting.  We let $(M, \omega, L)$
be a compact polarized K\"ahler manifold. That is, $L \to M$ is a holomorphic line bundle equipped with a
Hermitian metric $h$ whose curvature form $\Theta_h$ equals $i \omega$. Thus, $\omega \in H^{1,1}(M, 2 \pi \Z)$.
We also denote the kth tensor power of $L$ by $L^k$ and denote the space of holomorphic sections by
$H^0(M, L^k)$; we also put $d_k = \dim H^0(M, L^k)$.
We refer to \cite{Z6, Z7} for  background.

We further $\chi_1,
\chi_2$ be two quantizable symplectic diffeomorphisms of  $(M, \omega)$. The definition of quantizable is
from \cite{Z6,Z7} , which generalizations the implicit standard notion for special cases such as cat maps on a complex one dimensional torus (elliptic
curve). Namely, $\chi_j$ are symplectic diffeomorphisms which possess lifts as contact transformations
of the unit circle bundle $X_h = \partial D^*_h$ where $D^*_h $ is   the unit co-disc bundle in the dual line
bundle $L^*$ of $L$ with respect to the dual metric $h$. An exposition of the key notions K\"ahler quantization
can be found \cite{Gu}.

We  let $\{U_{\chi_1,k} \}_{k = 1}^{\infty}$ denote the semi-classical quantization of
$\chi_1$ as a sequence of  unitary operators on the Hilbert spaces $H^0(M, L^k)$.  Thus, $F = U_{\chi_2,k}$
denotes the quantum map quantizing $\chi_2$. As discussed in \cite{Z6,Z7} the 
quantizations have the form $U_{\chi, k}  :=\Pi_{h^k} \sigma_{k, \chi}  T_{\chi} \Pi_{h^k}$ where
$\Pi_{h^k}: L^2(M, L^l) \to H^0(M, L^k)$ is the orthogonal projection (Szeg\"o kernel),
where $T_{\chi}$ is the translation operator by $\chi$ and where $\sigma_{k, \chi}$ is a symbol
designed to make $U_{\chi, k}$ unitary.  More precisely, $T_{\chi}$ is the translation
operator by the lift of $\chi$ to $X_h$.

\begin{mainprop} \label{KAHLERintro} Let $\phi_{k, j}$ denote the eigensections of
$U_{\chi_1, k}$.  Suppose that $\chi_1, \chi_2$ almost nowhere commute in the
sense that the set  $\{z \in M: \chi_1 \chi_2(z) = \chi_2 \chi_1(z)\}$ has
measure zero.  Then $$\frac{1}{d_k} \sum_{j = 1}^{d_k} |\langle
U_{\chi_2, k} \phi_{k, j}, \phi_{k, j} \rangle|^2 \to 0. $$  \end{mainprop}

As a simple example, suppose that $A_1, A_2$ are two non-commutating elements of the
$\theta$ subgroup of $SL(2, \Z)$ (i.e. are congruent to the identity modulo 2). The associated symplectic maps
$ \chi_1, \chi_2 $  of $\R^2/\Z^2$ are then quantizable and  almost nowhere commute. So the eigenfunctions of one `quantum cat map' give rise to zero
quantum limits for the other. 

An interesting comparison to Proposition \ref{ANCsympcor}   is that the
semi-classical  eigensections $\phi_{N, j}$ do not appear to have
any localization properties which account for the almost
orthogonality of the matrix elements.

\subsection{Pointwise squares as matrix elements}


Let $\{\psi_{\hbar}\}$ be a semi-classical Lagrangian state, for instance a coherent state \cite{CR} in the Schr\"odinger
representation, or coherent states induced by a Bergman reproducing kernel, or the sequence of Gaussian beams
associated to a closed geodesic  \cite{R,BB}. In each case, we consider
the  norm-squares as matrix elements of semi-classical
Fourier integral operators,
\begin{equation} \label{NORMSQUARE} |\langle \psi_{\hbar}, \phi_j \rangle |^2 = \langle F_{\hbar} \phi_j, \phi_j \rangle, 
\;\;\; \mbox{where}\;\;  F_{\hbar}= \psi_{\hbar} \otimes \psi_{\hbar}^*. \end{equation}
The Schwartz kernel of $F_{\hbar}(x, y) = \psi_{\hbar} (x) \overline{\psi_{\hbar} (y)} $ inherits an oscillatory integral
representation from that of $\psi_{\hbar}$. 
One  relevant normalization is to take $||\psi_{\hbar}||_{L^2} = 1$.

The underlying Lagrangian or isotropic submanifold of $\{\psi_{\hbar}\}$ may or may not be invariant under the geodesic flow.
For instance, they are not invariant  for coherent states $\psi_{x, \xi}^{\hbar}$ (where the isotropic submanifold is a point), but they are   for the sequence of highest weight spherical harmonics $Y^k_k$ on the standard $S^2$ or for more general
Gaussian beams. However by an averaging argument (see \S \ref{MODSQ}) one can show that the weak* limits must be non-negative 
measures on the `orbit' of the underlying Lagrangian or isotropic submanifold under the geodesic flow. In the case
of a local canonical graph this `flowout' could be dense in $T^* M \times T^* M$ if $C$ is almost nowhere invariant
under $G^t \times G^t$ but for an isotropic submanifold the flowout can be a closed submanifold and one can
have non-trivial limits.

For the sake of brevity, we only  give explain how the  recent results in \cite{Z10} fit into the picture of
matrix elements of Fourier integral operators, which was not the approach used in that article. There are many
related examples that we will not consider here.

\subsection{\label{ALGEBRA} Idempotent canonical relations and algebras of Fourier integral operators}

There are other  settings where the  weak* limit problem is of interest. One is where
$I^0(M \times M, C)$ is a  $*$ algebra of Fourier integral operators. This occurs when
 the canonnical relation is idempotent in the sense that
$C^* = C = C^2. $ One such situation is the algebra of Fourier integral operators associated
to a symplectic cone $\Sigma \subset T^*X$. In fact $C$ need not be a Lagrangian submanifold.
It is sufficient that $C$ be isotropic (see \cite{W}). An example is when $\gamma$ is a closed
geodesic of a Riemannian manifold and where $\Sigma = \R_+ \dot{\gamma}$.  Then
$C$ is the diagional of $\Sigma \times \Sigma$.  The averaged
coherent state projections 
\begin{equation}\label{TAintro}  \langle \psi_{x, \xi}^{\hbar} \otimes (\psi_{x, \xi}^{\hbar})^* \rangle_L : =
\frac{1}{L} \int_{0}^L U^{t}  \left(\psi_{x, \xi}^{\hbar} \otimes (\psi_{x, \xi}^{\hbar}) \right)^* U^{-t} dt \end{equation}  are Toeplitz
operators in the case $\Sigma = \R_+ \dot{\gamma}$.

 Idempotent
canonical relations also  arise as leaf equivalence relations of null foliations of co-isotropic submanifolds
$\Sigma \subset T^* M$, also known as  
flowouts. 
In the case where the null-foliation is a fiber bundle with compact fiber over a leaf space
$S$ (a symplectic manifold), the algebra was denoted $\rcal_{\Sigma}$ and was studied
in \cite{GuSt}.  For the sake of brevity we omit further discussion and refer to   \cite{Z9} and to  \cite{ST,GU} for the study of quantum ergodicity in this setting.


\section{Background}

We recall that a Fourier integral operator $A: C^{\infty}(X) \to
C^{\infty}(Y)$ is an operator whose Schwartz kernel may be
represented by an oscillatory integral
$$K_A(x,y) = \int_{\R^N} e^{i \phi(x, y, \theta)} a(x, y, \theta) d\theta$$
where the phase $\phi$ is homogeneous of degree one in $\theta$.
The critical set of the phase is given by
$$C_{\phi} = \{(x, y, \theta): d_{\theta} \phi = 0\}. $$
When the map
$$\iota_{\phi} : C_{\phi} \to T^*(X, Y), \;\;\; \iota_{\phi}(x, y,
\theta) = (x, d_x \phi, y, - d_y \phi) $$ is an
 immersion the phase is called
non-degenerate.  Less restrictive is where the phase is clean, i.e.
$\iota_{\phi} : C_{\phi} \to \Lambda_{\phi} $, where
$\Lambda_{\phi} $ is the image of $\iota_{\phi}$,  is locally a
fibration with fibers of dimension $e$.  From \cite{HoI-IV}
Definition 21.2.5, the number of linearly independent
differentials $d \frac{\partial \phi}{\partial \theta}$ at a point
of $C_{\phi}$ is $N - e$ where $e$ is the excess.

We work in the polyhomogeneous framework of \cite{HoI-IV}, and asume that  classical poly-homogeneous symbols 
 $$ a(x, y, \theta) \sim \sum_{k=0}^{\infty} a_{-k}(x, y, \theta), \,\,(a_{-k} \; \mbox{ positive homogeneous of order
\; -k in }\;\theta.$$
 All of the results and notions of this note geeneralize to 
 semi-classical Fourier integral operators with semi-classical  symbols  $a \in S^{0,0}(T^*H \times (0,h_0]$
of the form
 $$ a_{\hbar} (s,\sigma) \sim \sum_{k=0}^{\infty} \hbar^k \;a_{-k}(s,\sigma), \,\,(a_{-k} \; \in  S_{1,0}^{-k}(T^* H)). $$ 
Since there is no essential difference in the weak* limit results in the two settings, we only consider the poly-homogeneous
one.

We a recall that the order of $F: L^2(X) \to L^2(Y)$ in the
non-degenerate case  is given in terms of a local oscillatory
integral formula by $m +
 \frac{N}{2} - \frac{n}{4},$, where $n = \dim X + \dim Y, $ where
$m$ is the order of the amplitude, and $N$ is the number of phase
variables in the local Fourier integral representation (see
\cite{HoI-IV}, Proposition 25.1.5); in the general clean case with
excess $e$, the order goes up by $\frac{e}{2}$ (\cite{HoI-IV},
Proposition 25.1.5'). Further, under clean composition of
operators of orders $m_1, m_2$, the order of the composition is
$m_1 + m_2 - \frac{e}{2}$ where $e$ is the so-called excess (the
fiber dimension of the composition); see \cite{HoI-IV}, Theorem
25.2.2.

The symbol $\sigma(\nu)$ of a Lagrangian (Fourier integral)
distributions is a section of the bundle $\Omega_{\half} \otimes
\mcal_{\half}$ of the bundle of half-densities (tensor the Maslov
line bundle). In terms of a Fourier integral representation it is
the square root $\sqrt{d_{C_{\phi}}}$ of the delta-function on
$C_{\phi}$ defined by $\delta(d_{\theta} \phi)$, transported to
its image in $T^* M$ under $\iota_{\phi}$. If $(\lambda_1, \dots,
\lambda_n)$ are any local coordinates on $C_{\phi}$, extended as
smooth functions in  neighborhood, then
$$ d_{C_{\phi}}: = \frac{|d \lambda|}{|D(\lambda,
\phi_{\theta}')/D(x, \theta)|}, $$ where $d \lambda$ is the
Lebesgue density.

\subsection{Local Weyl law for Fourier integral operators}

It was proved in \cite{Z2} (see also \cite{TZ,TZ2,DZ}) that if average the functionals
$\rho_j$, then the limit is a measure on the unit vectors in the
intersection $C \cap \Delta_{T^*M \times T^*M}$ of the canonical
relation $C$ with the diagonal in $T^*M \times T^*M$. That is, one has the local Weyl law,
\begin{equation} \label{LWL} \frac{1}{N(\lambda)} \sum_{j: \lambda_j \leq \lambda} \rho_j \to
\rho_{LWL},  \end{equation}  where the local Weyl law measure is given by 
$$\rho_{LWL}(F) = \int_{S (C \cap \Delta_{T^*M \times T^*M})}
\sigma_F d\nu, $$ where $d\nu$ is a `half-density  measure' and $S
(C \cap \Delta_{T^*M \times T^*M})$ is the set of unit covectors
in $ C \cap \Delta_{T^*M \times T^*M}$. In the case where $C$ is a
local canonical graph, this intersection is the fixed point set of
the correspondence $\chi$ and we write it as $S Fix(\chi)$.  

Note that the trace operation concentrated the average of the $\rho_j$ on the diagonal part of $C$. 
The individual matrix elements do not have this property, even though $\langle F \phi, \phi \rangle
= Tr F \phi \otimes \phi^*$ is a trace.

\section{Invariant states on $I^0(M \times M, C)$: Proof of Propositions \ref{w*L} and \ref{1}}

\subsection{Fourier integral operators associated to local canonical graphs}

The proof of Propositio \ref{1} is similar to the case of pseudo-differential operators proved in \cite{W}.

We first recall 
\begin{theo} (\cite{HoI-IV}, Theorem 25.3.1) If $C$ is a local canonical graph and $A \in I^0(M \times M, C)$,
then $A: L^2(M) \to L^(M)$ is bounded, and it is compact if the symbol of $A$ tends to 0 as $|\xi| \to \infty$. 
\end{theo}

We then prove
\begin{lem}  If $C$ is a local canonical graph and $A \in I^0(M \times M, C)$ then
\begin{equation} \label{NORM}   \sup_{S^* M}  |\sigma_A| = \inf_{K} ||A + K||.  \end{equation} \end{lem}

\begin{proof} The equality \eqref{NORM} is well known for $A \in \Psi^0(M)$. To generalize it to Fourier
integral operators associated to local canonical graphs it suffices to use that,
 in a sufficiently small cone,  $C$ is the graph of a canonical transformation. Then as in the proof of \cite{HoI-IV},
Theorem 25.3.1, 
$A^* A \in \Psi^0(M)$ and $\sigma_{A^* A} = |\sigma_A|^2$. It follows that
$$\sup_{T^* M} |\sigma_A|^2 = \inf_{K \; \mbox{compact}} ||A^* A + K| |=  \inf_{K \; \mbox{compact}} ||A+ K| |^2. $$ 
Here we use that for any $u \in L^2$ with $||u|| = 1$,
$||(A + K) u||^2 = \langle (A^* A + K_1) u, u \rangle $ for another compact operator $K_1$ and that
$ ||A^* A + K| | = \sup_{||u||=1} |\langle (A^* A + K_1) u, u \rangle| $ when $K$ is self-adjoint.
\end{proof}

If $K$ is any compact operator on $L^2(M)$ then $\langle K \phi_j, \phi_j \rangle \to 0$. Indeed, $\phi_j \to 0$ weakly
in $L^2$ and so $K \phi_j \to 0$ in norm. The principal symbol of $F$ determines
$F$ up to an element of $I^{-1}(M \times M, C)$ and the operators in this class are compact.  This proves
Proposition \ref{w*L}.

We now complete the proof of Propositions \ref{w*L} - \ref{1}.

\begin{proof}

For any compact operator $K$,  $\langle K \phi_j, \phi_j \rangle
\to 0$. Hence,  any limit of $\langle A \phi_k, \phi_k \rangle$ is
equally a limit of  $\langle (A + K) \phi_k, \phi_k \rangle$. By
the norm estimate, the limit  is bounded by $\inf_K ||A + K||$
(the infimum taken over compact operators).   Hence any weak limit
is bounded by a constant times $||\sigma_A||_{L^{\infty}} $ and is
therefore continuous on $C(S^*M)$. It is a positive functional
since each $\rho_j$ is and hence any limit is a probability
measure.

To prove the invariance of the limit measure, we apply an Egorov type theorem to $U^{-t} F U^t$
for $F \in I^0(M \times M, C)$ and for fixed $t$. The canonical relation of the composition is the composition
$$\Gamma_t^* \circ C \circ \Gamma_t = C,$$
where
$$\Gamma_t = \{( x, \xi, G^t(x, \xi)): |\tau| + |\xi| = 0\} .$$
Hence only the symbol $\sigma_F$ is changed. If we choose a nowhere vanishing half-density on $C$ (e.g. the
graph half-density corresponding to the symplectic voluime density on $T^*M$), then $\sigma_F$ may be
identified with a scalar function and the composite symbol is its pull-back under $G^t$. 

 By  invariance of the $\rho_k$,
any limit of $\rho_k(A)$ is a limit of $\rho_k(Op (\sigma_A \circ
\Phi^t))$ and hence the limit measure is invariant. It is also
time-reversal since the eigenfunctions are real-valued, i.e.
complex conjugation invariant.
\end{proof}

\subsection{Off-diagonal matrix elements \label{OFFDIAG}}

A similar argument applies to off-diagonal  matrix elements 
\begin{equation} \label{rhojk} \rho_{jk} (A) = \langle
 A
 \phi_j, \phi_k\rangle. \end{equation}
The discussion is very similar to that in \cite{Z5} in the pseudo-differential case.

\begin{prop} \label{OD}Let $\rho_{\infty}$ be a limit of the sequence of functionals \eqref{rhojk} with
the gap $\lambda_j - \lambda_k \to \tau$
on  $I^0(M, \times M, C)$. Suppose that the canonical relation $C$
is invariant under the geodesic flow $G^{-t} \times G^t$. Then  $\rho_{\infty}$ is a  signed $G^t$-eigennmeasure of mass
$\leq 1$ on $C$.
\end{prop}

The only change to the proof of Proposition \ref{1} is in the last step. The functionals $\rho_{jk}$ are no longer
invariant but rather satisfy \eqref{QINVoff}. It follows that
if $\lambda_j - \lambda_k \to \tau$ then any limit measure is a $G^t$ eigenmeasure on $C$ with eigenvalue
$e^{it \tau}$.

\section{\label{AO} Almost orthogonality: Proof of Proposition \ref{ANCsympcor}}

This section is motivated by the proof of the quantum ergodic restriction theorem in \cite{TZ,TZ2} (see also \cite{DZ}).
The criterion for QER in those papers is an almost nowhere commutativity condition between two canonical transformations,
or equivalently an almost nowhere invariance problem. It is  not clear that the condition for QER in those papers
is sharp.  

We now prove Proposition \ref{ANCsympcor}.

\begin{proof} 

We are assuming that  for $t > 0$, the
set $\ccal_{t} = \{(x, \xi) \in S^* M: G^t \chi = \chi G^t\}$ has
Liouville measure zero.

It suffices to show that
 \begin{equation} \label{VARIANCE} \frac{1}{N(\lambda)} \sum_{\lambda_j \leq \lambda}
   |\langle F
    \phi_{\lambda_j}, \phi_{\lambda_j} \rangle  |^{2} = o(1) \,\,\, \text{as} \,\, \lambda \rightarrow \infty. \end{equation}
We put
$$F(t) = U^{t*} F U^t, \;\; \langle F \rangle_T : = \frac{1}{T} \int_{-T}^T F(s) ds, \;\; \mbox{where}\;\; U^t = e^{i t \sqrt{\Delta}}. $$
Then $\langle F(t)
    \phi_{\lambda_j}, \phi_{\lambda_j} \rangle = \langle F
    \phi_{\lambda_j}, \phi_{\lambda_j} \rangle. $

For any operator $A$ we have $$|\langle A \phi_{\lambda},
\phi_{\lambda} \rangle|^2 \leq \langle A^* A \phi_{\lambda},
\phi_{\lambda} \rangle. $$ It follows that
$$|\langle F
    \phi_{\lambda_j}, \phi_{\lambda_j} \rangle  |^{2} \leq \langle \langle F \rangle_T^* \langle F \rangle_T \phi_{\lambda},
\phi_{\lambda} \rangle. $$ Hence,
 \begin{equation} \label{VARIANCEb} \frac{1}{N(\lambda)} \sum_{\lambda_j \leq \lambda}
   |\langle F
    \phi_{\lambda_j}, \phi_{\lambda_j} \rangle  |^{2} \leq \frac{1}{N(\lambda)} \sum_{\lambda_j \leq \lambda}
    \langle \langle F \rangle_T^* \langle F \rangle_T \phi_{\lambda},
\phi_{\lambda} \rangle. \end{equation} We now let  $\lambda \to
\infty$ and use the local Weyl law \eqref{LWL} for Fourier integral operators.
We have,
$$ \langle F \rangle_T^* \langle F \rangle_T  = \frac{1}{T^2}
\int_{- T}^T \int_{-T}^T F(s)^* F(t) ds dt. $$ Since we are taking
a trace, we can cycle the $U^t$ to get
$$ \langle F \rangle_T^* \langle F \rangle_T  = \frac{1}{T^2}
\int_{- T}^T \int_{-T}^T U(t - s)^* F^* U(t - s) F  ds dt. $$ We
change variables to $u = \frac{t - s}{2}, v = \frac{t + s}{2} $
and simplify to get
$$ \langle F \rangle_T^* \langle F \rangle_T  = \frac{1}{T}
\int_{-T}^T U(t)^* F^* U(t) F  \rho_T(t) dt. $$ Here, $\rho_T(t)$
is the measure in $[-T, T] \times [-T, T]$ of $\{(s, s'): s - s' =
t\}$. For each $t$ the local Weyl law gives
\begin{equation} \label{VARIANCEc}  \frac{1}{N(\lambda)} \sum_{\lambda_j \leq \lambda}
    \langle   U(t)^* F^* U(t) F  \phi_{\lambda},
\phi_{\lambda} \rangle \to \int_{S Fix G^{-t} \chi^{-1} G^t \chi}
\sigma_{t, F},
\end{equation}
where $\sigma_{t, F}$ is a composite density on the fixed point
set. The fixed point set is exactly the set where $\chi G^t = G^t
\chi$ and has measure zero for all $t \not= 0$. Hence the integral
tends to zero for all $T > 0$.

\end{proof}

This arguments works too well because the assumption is so strong.
A related argument is to just assume that there exists $t_0$ so
that $\chi G^{n t_0} = G^{n t_0} \chi$ only holds on a set of
measure zero. It is this argument which was in effect used in \cite{TZ,TZ2}.

\begin{prop} \label{ANCsympcora}Let $F$ be a Fourier integral operator associated to
a symplectic correspondence $\chi$. Assume that there exists $t_0 \not= 0$ so that for  $n = 1, 2, 3 \dots$, the
set $\ccal_{t} = \{(x, \xi) \in S^* M: G^{n t_0} \chi = \chi G^{n t_0}\}$ has
Liouville measure zero. Then there is a density one subsequence of
eigenfunctions so that $\langle F \phi_{\lambda_j},
\phi_{\lambda_j} \rangle \to 0$. \end{prop}

\begin{proof}

We define
$$\langle F \rangle_M : = \frac{1}{2 M} \sum_{m = - M}^M U^{- m
t_0} F U^{m t_0}. $$ Going through the same argument gives the
upper bound
\begin{equation} \label{VARIANCEbc} \frac{1}{N(\lambda)} \sum_{\lambda_j \leq \lambda}
   |\langle F
    \phi_{\lambda_j}, \phi_{\lambda_j} \rangle  |^{2} \leq \frac{1}{N(\lambda)} \sum_{\lambda_j \leq \lambda}
    \langle \langle F \rangle_M^* \langle F \rangle_M \phi_{\lambda},
\phi_{\lambda} \rangle. \end{equation} We then have
$$\begin{array}{lll} \langle F \rangle_M^* \langle F \rangle_M &= &\frac{1}{M^2} U^{t_0
(m - n)*} F^* U^{t_0 (m - n))} F \\ && \\
& = & \frac{1}{M} \sum_{p = - M}^M \frac{\# p}{M} U^{t_0 p *} F^*
U^{t_0 p} F,  \end{array} $$ where $\# p = \# \{(m, n) \in [- M, M
] \times [- M, M ]: m - n = p \}. $ We then apply the local Weyl
law and find that the only term which makes a non-vanishing
contribution is $p = 0$. So it is $O(\frac{1}{M})$.

\end{proof}

\subsection{\label{Almostdisjoint} Almost disjoint energy surfaces}

The commutator $[\sqrt{\Delta}, F]$ is always  of order one if $F \in I^0(M \times M, C)$. 
The symbol of $\sqrt{\Delta} F$ at $(x, \xi, y, \eta)$ is the
product $|\xi|_x \sigma_F$ while in the other order we have
$|\eta|_y \sigma_F$. So they do not cancel unless $\chi$ preserves
$S^*_g X$, which is not the case when they almost never commute.

\begin{lem} \label{chigtNC}  Suppose that $\chi$ and $G^t$ almost never commute.
Then $\chi(S^*_g M) \cap S^*_g M$ has Liouville measure zero. \end{lem}

Let $H(x, \xi) = |\xi|_g$ be the metric norm function. The
Hamiltonian flow of $\chi^* H$ is $\chi G^t \chi^{-1}$. The orbit
of $(x, \xi)$ under this flow is almost certainly disjoint from
the $G^t$ orbit of $(x, \xi)$. Also the Hamiltonian vector field
$\xi_H$ of $H$ almost certainly satisfies $\chi_* \xi_H \not=
\xi_H$.  If $\chi^*H = H$ on an open set, then we take the
Hamiltonian vector field of both sides to get $\chi_* \xi_H =
\xi_H$. Similarly for any set of positive measure.

We can then give a second proof of Proposition \ref{ANCsympcor}. 

\begin{proof}  It is well-known that $\Delta$-eigenfunctions concentrate microlocally on 
the energy surfaces $|\xi| = E$ (see \cite{Zw},  Theorem 6.4).  In the homogeneous setting we may identify  all
the energy surfaces
with $S^* M$.  It follows that  $F \phi_j$  microlocally concentrates on $\chi(S^* M)$. In other words
we may construct semi-classical cutoffs $Op_{\hbar}(b)$ to $S^* M$ and $Op_{\hbar}(\chi^* b)$ so
that $Op_{\hbar} \phi_j = \phi_j + \ocal(\hbar), Op_{\hbar}(\chi^* b) F \phi_j = F \phi_j + \ocal(\hbar). $
By 
Lemma \ref{chigtNC}, the intersection has Liouville measure zero. It follows that
$$ \langle F \phi_j, \phi_j \rangle = \langle Op_{\hbar}(\chi^* b) F \phi_j,
Op_{\hbar}(b)  \phi_j  \rangle + \ocal(\hbar)  \leq || Op_{\hbar}(\chi^* b)   Op_{\hbar}(b) 
\phi_j|| + \ocal(\hbar). $$
But $Op_{\hbar}(\chi^* b)   Op_{\hbar}(b)  = Op_{\hbar}(b \chi^* b) + \ocal(\hbar)$. Then 
$$\begin{array}{lll} \frac{1}{N(\lambda)} \sum_{\lambda_j \leq \lambda}
   |\langle F
    \phi_{\lambda_j}, \phi_{\lambda_j} \rangle  |^{2}  & \leq & \frac{1}{N(\lambda)} \sum_{\lambda_j \leq \lambda}
   |\langle Op_{\hbar}(b \chi^* b)
    \phi_{\lambda_j}, \phi_{\lambda_j} \rangle  |^{2}  + \ocal(\lambda^{-1}) \\&&\\
&\leq &  \frac{1}{N(\lambda)} \sum_{\lambda_j \leq \lambda}
   \langle \left(Op_{\hbar}(b \chi^* b)^* p_{\hbar}(b \chi^* b \right)
    \phi_{\lambda_j}, \phi_{\lambda_j} \rangle  |^{2}  + \ocal(\lambda^{-1}) \\&&\\& \to &  \int_{S^* M} b^2 \chi^* b^2 \;\; \mbox{as}\;\; \lambda \to \infty.  \end{array}$$
But  by assumption, for any $\epsilon > 0$ one may construct $b$ so that  the support of $b \times \chi^* b$ has volume $\leq \epsilon$. It follows that  $  \int_{S^* M} b^2 \chi^* b^2 \leq \epsilon$,
and therefore $\frac{1}{N(\lambda)} \sum_{\lambda_j \leq \lambda}
   |\langle F
    \phi_{\lambda_j}, \phi_{\lambda_j} \rangle  |^{2}  \to 0$.

\end{proof}

The proof indicates that although a subsequence of density one of the matrix elements tends to zero,
not all them need to.  It could be the case that a subsequence of eigenfunctions
$\phi_{j_k}$ concentrates microlocally on a closed geodesic  $\gamma$ and  that $\chi(\gamma) = \gamma$.
Then even though $\chi(S^* M) \cap S^* M$ has Liouville measure zero, the full sequence
 of  matrix elements $\langle F \phi_j, \phi_j \rangle$
need not tend to zero.

\begin{rem} It is natural to ask how sharp the almost nowhere invariance or commutation conditions
are in the proof of almost orthogonality.  That is, we ask whether the following converse
to Proposition \ref{ANCsympcor} is true:
\medskip

\noindent{\bf Question:}  Denote the canonical relation
of $F$ by $C$. Then in the notation of this section, do we have,  
$$ \frac{1}{N(\lambda)} \sum_{\lambda_j \leq \lambda}
   |\langle F
    \phi_{\lambda_j}, \phi_{\lambda_j} \rangle  |^{2}  \to 0   \implies  \mbox{C is almost nowhere invariant under}\;\; g^t? $$


\end{rem}

\subsection{\label{QM} Almost nowhere commuting quantum maps on K\"ahler manifolds}

We now prove the analogue, Proposition \ref{KAHLERintro},  in the \kahler setting. Let $\chi_1,
\chi_2$ be two quantizable symplectic diffeomorphisms of a compact
\kahler manifold $(M, \omega)$ which almost nowhere commute in the
sense that the set where $\chi_1 \chi_2 = \chi_2 \chi_1$ has
measure zero. We then let $U_{\chi_1}$ denote the quantization of
$\chi_1$ as a unitary operator on $H^0(M, L^k)$ and we let $F$
denote any quantum map quantizing $\chi_2$.


An interesting comparison to the previous case is that the
semi-classical  eigensections $\phi_{k, j}$ do not appear to have
any localization properties which account for the almost
orthogonality of the matrix elements.
\bigskip

\noindent{\bf Proof of Proposition \ref{KAHLERintro}}
\bigskip

We again consider the partial time average
$$\langle F \rangle_M : = \frac{1}{2 M} \sum_{m = - M}^M U^{- m
} F U^{m }. $$ Going through the same argument gives the upper
bound
\begin{equation} \label{VARIANCEbcb} \frac{1}{N(\lambda)} \sum_{\lambda_j \leq \lambda}
   |\langle F
    \phi_{\lambda_j}, \phi_{\lambda_j} \rangle  |^{2} \leq \frac{1}{N(\lambda)} \sum_{\lambda_j \leq \lambda}
    \langle \langle F \rangle_M^* \langle F \rangle_M \phi_{\lambda},
\phi_{\lambda} \rangle. \end{equation} We then have
$$\begin{array}{lll} \langle F \rangle_M^* \langle F \rangle_M &= &\frac{1}{M^2} U^{t_0
(m - n)*} F^* U^{t_0 (m - n))} F \\ && \\
& = & \frac{1}{M} \sum_{p = - M}^M \frac{\# p}{M} U^{t_0 p *} F^*
U^{t_0 p} F,  \end{array} $$ where $\# p = \# \{(m, n) \in [- M, M
] \times [- M, M ]: m - n = p \}. $ We then apply the local Weyl
law and find that the only term which makes a non-vanishing
contribution is $p = 0$. So it is $O(\frac{1}{M})$ and therefore the limit equals zero. QED

\section{\label{MODSQ} Modulus squares as matrix elements }

In this section, we tie together the weak * limit problem for Fourier integral operators with some recent results on  modulus squares $|\phi_j^{\C}(z)|^2$ of analytic continuations
of eigenfunctions to Grauert tubes and their restrictions to geodesic arcs  (cf. \cite{Z8,Z10}), and to some related pointwise Weyl laws for coherent state projections in \cite{PU}.  
The relevant matrix elements 
differ from the preceding ones in that 
the underlying canonical or wave front relations  are not local canonical graphs. The results
are correspondingly different: the canonical relations are not invariant under the geodesic
flow, but Proposition \ref{ANCsympcor} is false for them. When properly normalized the
weak* limits can be non-zero. In fact, we will use an averaging argument or a flowout construction
  to make the canonical
relation geodesic flow invariant when $(x, \xi)$ is a periodic point. This is impossible 
for a local canonical graph which is almost nowhere invariant.

   \subsection{\label{GT} $|\phi_j^{\C}(z)|^2$ on a Grauert tube}

We first consider pointwise modulus squares of $|\phi_j^{\C}(z)|^2$ of analytic continuations of eigenfunctions
to Grauert tubes in the complexification of $M$.  
We refer to \cite{Z8, Z10} and their references for background on the analytic continuation
and the geometry of Grauert tubes.

As discussed at length in \cite{GS1,Z8},  a Grauert tube $M_{\epsilon}$  is a strictly pseudo-convex domain
in the complexification $M_{\C}$ of a real analytic Riemannian mannifold $(M, g)$. Its defining
function $\rho(z)$ is the analytic continuation of the squared  distance function $r^2(x, y)$
to $x = z, y = \bar{z}$. The Grauert tube function is its square root $\sqrt{\rho}(z)$ (we are 
ignoring here some constants). 
A key point is that  is the image of a ball bundle $B^*_{\epsilon} M$
under the imaginary time  exponential map $E: B_{\epsilon}^* M \to M_{\epsilon}$, $E(x, \xi) = \exp_x i \xi$. 
The Grauert tube function $\sqrt{\rho} $ corresponds to  the norm function $|\xi|_g$ of the metric under $E$. 
Moreover $E$ conjugates the geodesic flow to a Hamiltonian  flow  $g^t$ on $ M_{\tau}$ with respect to its
adapted K\"ahler form.

The principal Fourier integral operator in this context is the Poisson kernel $P^{\epsilon}(z, y)$
on $\partial M_{\epsilon} \times M$, defined as follows:
The wave group of $(M, g)$ is the unitary group $U(t) = e^{ i
 t \sqrt{\Delta}}$. Its kernel $U(t, x, y)$ solves the  `half-wave equation',
\begin{equation} \label{HALFWE} \left(\frac{1}{i} \frac{\partial }{\partial t} -
\sqrt{\Delta}_x \right) U(t, x, y) = 0, \;\; U(0, x, y) =
\delta_y(x). \end{equation}
The Poisson-wave kernel $P^{\tau}(z,y)$  is the analytic continuation $U( i \tau, x, y)$  of the wave kernel  with respect to time,  $ t \to  i \tau\in \R \times \R_+$ and then in $x$, i.e.
\begin{equation} \label{PK} P^{\tau}(z, y) = U(i \tau, z, y). \end{equation} Thus, the Poisson- kernel  has the eigenfunction expansion for $\tau > 0$,
\begin{equation}\label{POISEIGEXP}  U (i
\tau, x, y) = \sum_j e^{- \tau \lambda_j} \phi^{\C}_{\lambda_j}(z)
\phi_{\lambda_j}(y), \;\;\; z \in \partial M_{\tau}, y \in M.
\end{equation}
Thus,
$P^{\tau} \phi_j = e^{- \tau \lambda_j} \phi_j^{\C}. $

The key fact is that $P^{\tau}(z, y)$ is  a Fourier integral operator with complex phase.
It is adapted to the symplectic isomorphism  of $T^* M$ with the symplectic cone $\Sigma_{\tau}
\subset T^* \partial M_{\tau}$ generated by the contact form $\alpha_z$.  This result 
was stated by Boutet de Monvel in \cite{Bou}  but only recently have detailed proofs been published
(cf.  \cite{Z8,St}).  The main ingredient is the  analytic continuation of the wave kernel
due to  J. Hadamard.   

If we define
\begin{equation} \Phi^{\tau}_z(y) = P^{\tau}(z, y)  \in L^2(M),  \end{equation} 
then 
\begin{equation} \label{PTAUME} e^{- \tau \lambda_j} \phi_j^{\C}(z)  = \langle \Phi^{\tau}_z, \phi_j \rangle_{L^2(M)}. \end{equation}
It follows that 
\begin{equation} \label{STATE} e^{- 2 \tau \lambda_j}  |\phi_j^{\C}(z)|^2  = |\langle \Phi^{\tau}_z, \phi_j \rangle|^2
= \langle \Phi^{\tau}_z \otimes (\Phi^{\tau}_z)^* \phi_j, \phi_j \rangle,  \end{equation}
where $\Phi^{\tau}_z \otimes (\Phi^{\tau}_z)^* $ is the rank one projector onto $\Phi^{\tau}_z(y)$. 
Since
$P^{\tau}(z, y)$ is a homogeneous Fourier integral kernel with  complex phase on $\partial M_{\tau} \times M$,
the Schwartz kernel  $\Phi^{\tau}_z(y)  \overline{\Phi^{\tau}_z(y')} $  of $ \Phi^{\tau}_z \otimes (\Phi^{\tau}_z)^* $ 
is also a Fourier integral kernel with complex phase on $\partial M_{\tau} \times  M \times M$. If we  fix $z$
then $\Phi^{\tau}_z(y)  \overline{\Phi^{\tau}_z(y')} $   is a  Fourier integral kernel with complex phase on $ M \times M$.
The associated `canonical relation' is actually an isotropic relation. It  is a product relation of the form,
\begin{equation} \label{LAMBDAZ} \Lambda_z \times \Lambda_z^* \subset T^* M \times T^* M,  \end{equation}
where
$$\Lambda_z = \{(x, \xi): E(x, \xi) = z\}. $$
Fourier integral operators associated to isotropic relations were introduced in \cite{W}.

One of the
main problems with the  `states'
\eqref{STATE} is that they are not normalized. It is proved in \cite{Z8}  (Corollary 2) that 
$e^{- 2 \epsilon \lambda_j}  |\phi_j^{\C}(z)|^2 \leq C \lambda^{m - 1} $. The upper bound is sharp (it is
attained by highest weight spherical harmonics on the standard $S^m$) but it is generally not attained on a
generic analytic Riemannian manifold, nor is it attained at general points $z$ even when it is attained at
some point; and even when it is attained   at  some  $z$, it is  only attained by a sparse subsequence. 

To deal with these issues, we first observe that  the  diagonal matrix elements of $\Phi^{\tau}_z \otimes (\Phi^{\tau}_z)^*$ are the same as for the partial time averages
\begin{equation}\label{TA}  \langle \Phi^{\tau}_z \otimes (\Phi^{\tau}_z)^* \rangle_T : =
\frac{1}{2 T} \int_{-T}^T U^{t}  \left(\Phi^{\tau}_z \otimes (\Phi^{\tau}_z) \right)^* U^{-t} dt. \end{equation} 
The full time average is the limit $T \to \infty$ in \eqref{TA}  in the weak operator topology.
In general,  the Fourier integral properties of 
$\Phi^{\tau}_z \otimes (\Phi^{\tau}_z)^*$ are destroyed by infinite time averaging, or at best
they are unclear. 
But if
$z$ is a periodic point for $g^t$ of period $L$, then $\langle \Phi^{\tau}_z \otimes (\Phi^{\tau}_z)^* \rangle_L$
is a Fourier integral operator with complex phase on $M $ which commutes to leading order with $\sqrt{\Delta}$, i.e.
$$\begin{array}{lll} [ \sqrt{\Delta}, \langle \Phi^{\tau}_z \otimes (\Phi^{\tau}_z)^* \rangle_L] & = &
\frac{1}{2 L} \int_{-L}^L  \frac{d}{dt} U^{t}  \left(\Phi^{\tau}_z \otimes (\Phi^{\tau}_z) \right)^* U^{-t} dt  \\&&\\
& = & \frac{1}{2 L} \left(U^{L}  \left(\Phi^{\tau}_z \otimes (\Phi^{\tau}_z) \right)^* U^{-L}
-   \left(\Phi^{\tau}_z \otimes (\Phi^{\tau}_z) \right)^* \right). \end{array} $$
So for a periodic point, we set $T$ equal to an integer multiple of the periodic and  substitute the time average into \eqref{PTAUME}. Geometrically, this corresponds to replacing 
 \eqref{LAMBDAZ} by its flowout  
\begin{equation} \label{FLOWOUT} \bigcup_{t \in [0, L]}  \Lambda_{g^t(z)} \times \Lambda_{g^t(z)} = \gamma \times \gamma, \end{equation}
where $\gamma$ is the periodic orbit of $z$. Obviously it is   invariant under
$g^t \times g^t$. 

Secondly, we pull back $\phi_j^{\C}$ under a parametrization $\gamma^{\tau}(t)$ of
the $g^t$ orbit of $z$ on $\partial M_{\tau}$. That is, we restrict the complexified eigenfunction
to the orbit of $z$. We then  re-normalize $\gamma^{\tau *} e^{- 2 \tau \lambda_j}  |\phi_j^{\C}(z)|^2 $ to have mass one
on $\gamma$. Here we use the orbit of $g^t$ on $\partial M_{\tau}$ and define
\begin{equation} \label{Uj} U_j^{\tau} (t ) =  \frac{ \phi_j^{\C}(\gamma^{\tau}(t))}{|| \gamma^{\tau *} \phi_j^{\C}\||_{L^2([0, L])}}. 
\end{equation}

We then have 
\begin{prop} \cite{Z10} (Proposition 2)  \label{PERORB} If $z$ is a periodic point, and if $U_j$
does not vanish identitically on the orbit of $z$, then the unique weak * limit of the positive unit
mass measures $\{|U_j|^2\}$ on the orbit $\gamma$ of $z$ in $\partial M_{\tau}$  is 
$\frac{1}{L} dt$, the normalized periodic orbit measure on $\gamma.$
Equivalently, $\{|U_j|^2\} \to dt$ weakly as $j \to \infty$.
 \end{prop}

The proof is given in \cite{Z10}.  One proves an
 Egorov type theorem in the class of Toeplitz operators. It shows that the weak* limits
must be invariant under $g^t$.  Since they are invariant probability measures on $\gamma$
the only possible limit is $\frac{1}{L} dt.$ 

In \cite{Z8}, an asymptotic formula for the average over the spectrum of the matrix elements
\eqref{STATE}  is given. It involves the stability matrix of the geodesic flow along the orbit
of $z$. However, this data cancels when we take the quotient with $|| \gamma^{\tau *} \phi_j^{\C}\||_{L^2([0, L])}$. 

One may understand Proposition \ref{PERORB} in terms of    time averages of the Fourier integral operator
$\Phi_z^{\tau} \otimes \Phi_z^{\tau*}$.  The 
 leading order  evolution $U^{t}  \Phi^{\tau}_z \otimes (\Phi^{\tau}_z)  U^{-t}$ is somewhat
analogous to the evolution of the standard
 coherent states  \begin{equation} \label{CS} \psi^{\hbar}_{x, \xi}(y) = 2^{-n/4} (2 \pi \hbar)^{- \frac{3n}{4}} e^{- i \frac{\xi \cdot x}{2 \hbar}} 
e^{i \frac{\xi \cdot y}{\hbar}} e^{- \frac{(x - y)^2}{2 \hbar}}  \end{equation} 
on $\R^n$, although it is closer to that of coherent states in the Bargmann-Fock representation. 
The evolution of coherent states has been studied extensively
 in various settings. The model case  of
evolution of 
standard 
coherent states in the Schr\"odinger or Bargmann-Fock representations under linear
Hamilton flows  is discussed
in detail in \cite{CR}, and the evolution of  coherent states on manifolds are discussed in   \cite{PU}. The
case relevant here is that of the Poisson FBI transform to Grauert tubes and the
the Poisson  coherent
states are discussed in \cite{Z8,Z10}. In each case, to  leading order the projection evolves
as that of a distorted coherent state projection
$\Phi^{\tau}_{g^L(z)} \otimes (\Phi^{\tau}_{g^L(z)} )^*$ `centered' on the orbit of $z$. The
shape distortion is due to the Jacobi stability matrix (i.e. $D g^t$ along the orbit).  Thus, 
the time average \eqref{TA} with $T = L_{\gamma}$ (the period of the orbit) is a Fourier
integral operator with complex phase space with wave front set along $\gamma \times \gamma$
and with symbol constant along the orbit. 


It would be interested to find the optimal  analogue of Proposition \ref{PERORB} when
$z$ is not a periodic point, e.g. when it is a regular point for $g^t$ and when $g^t$ is ergodic
on $\partial M_{\tau}$. This case is also studied in \cite{Z10}. The problem is that one
cannot $L^2$ normalize on the entire orbit unless one uses a weight, e.g. the characteristic
function of an interval. But it is apriori possible that the local $L^2$ norms are incommensurable
along the orbit. This cannot happen over fixed compact sets but the norms of the restrictions
could have different orders of magnitudes on parameter intervals for $\gamma: \R \to \partial M_{\tau}$ which are separated by an amount greater than the Ehrenfest time $C \log \lambda$.



\subsection{Coherent states}

A variation on the preceding example is to study matrix elements for  coherent
state projectors, where the coherent states are defined by \eqref{CS}, or more generally
(as in \cite{PU}),
$$\psi^{a}_{x, \xi}(y) = \rho(x - y)  2^{-n/4} (2 \pi \hbar)^{- \frac{3n}{4}} e^{- i \frac{\xi \cdot x}{2 \hbar}} 
e^{i \frac{\xi \cdot y}{\hbar}} \hat{a} ( \frac{(x - y)}{\sqrt{ \hbar}}) $$
for  $a \in \scal(\R)$. The coherent states (or wave packet) transform is defined
by 
$f \to \langle f, \psi_{x, \xi}^{\hbar} \rangle. $ Like \eqref{PTAUME}, it is an  FBI transform but it
is  adapted to the heat kernel rather than the  Poisson kernel and has different inversion
properties. 

As above, 
 $\psi^a_{x, \xi} \otimes \psi^{a*}_{x, \xi}$ is a semi-classical  Fourier integral operator with
complex phase. 
If $(x, \xi)$ is a periodic point of
period $L$ we again consider,
$$\langle \psi^a_{x, \xi} \otimes \psi^{a*}_{x, \xi} \rangle_L : =  \frac{1}{2 L}\int_{-L} ^L  U^{t} (\psi^a_{x, \xi} \otimes \psi^{a*}_{x, \xi} ) U^{-t} dt. $$ As above, 
$\langle \langle \psi^a_{x, \xi} \otimes \psi^{a*}_{x, \xi} \rangle_L \psi_j, \psi_j \rangle = 
\langle \psi^a_{x, \xi} \otimes \psi^{a*}_{x, \xi}  \psi_j, \psi_j \rangle. $
To leading order in $\hbar$ it is the same as 
$$\langle \psi^a_{x, \xi} \otimes \psi^{a*}_{x, \xi} \rangle_L \simeq  \frac{1}{2 L}\int_{-L} ^L   \psi^{a_t}_{G^t(x, \xi)} \otimes \psi^{a_t*}_{G^t(x, \xi)}  dt, $$
where $a_t$ is a deformed symbol whose principal part is $a_0 \circ dg^t$ (cf. \cite{CR,PU}).


As in the setting of analytic continuation to Grauert tubes, we restrict  the diagonal matrix elements
$|\langle \langle \psi^a_{x, \xi} \otimes \psi^{a*}_{x, \xi} \rangle_L \psi_j, \psi_j \rangle |^2 $
to the orbit of $(x, \xi)$ and view them as a sequence of  positive measures. When $(x, \xi)$
is a periodic point, we restrict to the pull back to $[0, L]$ and divide by the mass to obtain
a sequence of probability measures on  $[0, L]$. 
The weak* limits are then constant multiples of Lebesgue measure as in Proposition \ref{PERORB}.

.

\section{Proof of Proposition \ref{SA}}

This additional symmetry of Proposition \ref{SA} comes from the fact  that $I^0(M \times M, C)$ is a right and left module over
$\Psi^0(M)$ with respect to composition of operators, and that on the symbol level the left and right compositions commute
We thus consider the case where $\rho_j(AF) \sim \rho_j(FA)$
where $A \in \Psi^0(X)$, e.g. if  $F^* = F$ as in the cases
$\cos t \sqrt{\Delta}$ or $T_p$, or if  $\phi_j$ is an eigenfunction
of $F$. We restate Proposition \ref{SA} and refer to the diagram \eqref{DIAGRAMintro}.

\begin{prop} \label{LR} Suppose that $\rho_j(AF) \sim \rho_j(FA)$ and   that the canonical relation $C$
is invariant under the geodesic flow $G^t$. Then the limit
measures are among the $G^t$-invariant signed  measures of mass
$\leq 1$ on $C$ satisfying $\rho_{\infty}(\pi_X^* a \sigma) =
\rho_{\infty}(\pi_Y^* a \sigma). $
\end{prop}

\begin{cor} Let $\iota: C \to C$ be the involution $\iota(\xi,
\eta) = (\eta, \xi)$. Let $\nu_{\infty}$ be the limit on $\Psi^0$.
Then $\iota^* \rho_{\infty} = \rho_{\infty}$. Equivalently,
$$\pi_{1*} \rho_{\infty} = \pi_{2*} \rho_{\infty} = \nu_{\infty}.
$$\end{cor}

In the case where $C$ is the graph of a canonical transformation
$\chi$, this says that if we identify $\rho_{\infty}$ with a
measure on the domain, then $\chi^* \rho_{\infty} =
\rho_{\infty}$. In the general case of a local canonical graph, it
does not say that $\rho_{\infty}$ is invariant  under the correspondence
but rather than it is invariant under each local branch. \medskip

\begin{proof}

  If we fix one (elliptic) element 
$F  \in I^0(M \times M, C)$ (i.e. with nowhere vanishing symbol) then   $I^0(M \times M, C)$ is
spanned by sums of operators $A F B$ with $A, B \in \Psi^0(M)$. Thus $\rho_j$ induces functionals on 
$\Psi^0(M) \times \Psi^0(M)$ of the form
\begin{equation}  \rho_j (A F B): = \langle A F B  \phi_j, \phi_j
\rangle, \end{equation} which can be normalized to have  mass one by
\begin{equation} \hat{\beta}_j (A F B): = \frac{ \langle A F B  \phi_j, \phi_j
\rangle}{\langle F \phi_j, \phi_j \rangle}, \;\;\; \mbox{if}\;\; F \phi_j \not= 0. \end{equation} Thus, we can
define right and left functionals
$$\beta_j^L (A) = \rho_j(AF), \;\;\ \beta_j^R(A) = \rho_j(FA) \;\; \mbox{on} \;\; \Psi^0(M)$$
and obtain signed limit measures  on $S^* M$ from the weak* limits. Then $\{\beta_j^L\}$ 
and $\{\beta_j^R\}$ have weak* limits
along the same subsequences and the limits are the same.

Then
$$\rho_{\infty}(\pi_X^* a \sigma) \sim \rho_j(A F) = \langle A F \phi_j, \phi_j \rangle = \overline{\langle F^*
A^* \phi_j, \phi_j \rangle} = \overline{\langle F A^* \phi_j,
\phi_j \rangle} \sim \pi_Y^* a \sigma $$ since $\sigma_{A^*} \sim
\overline{\sigma_A}$ and since $\overline{\sigma_F} = \sigma_F.$

\end{proof}

\begin{rem} When $F$ is associated to a symplectic diffeomorphism $\chi$ of $T^* M \backslash 0$, then the left and right functionals are
related as follows:

\begin{equation} \label{LEFTRIGHT} \beta_j^R(A) = \rho_j(FA) = \rho_j(FAF^{-1} F)=\beta_j^L (F^{-1} A F). \end{equation}
By Egorov's theorem $F^{-1} A F$ is a pseudo-differential operator with symbol $\sigma_A \circ \chi$, and so 
along any sequence with a weak* limit
\begin{equation} \label{LRLIMIT} \beta_{\infty}^R(\sigma_A)  = \beta_{\infty}^L(\sigma_A \circ \chi). \end{equation}
\end{rem}

\section{\label{HECKE} Isometries and Hecke operators}

In this section we consider the weak* limits in Proposition \ref{w*L}  and Corollary \ref{2} in the case of isometries
and sums of isometries known as  Hecke operators. By lifting the weak* limit problem to the canonical relation
instead of $S^* M$ and using Propositon \ref{LR},  we   obtain a new invariance principle. We also make concrete identifications of the canonical 
relations. 

In discussing the canonical relations of Hecke operators we make use of the well-known co-tangent lift
 $f_{\sharp}$ of a diffeomorphism $f: X_1 \to X_2$,  defined by
\begin{equation} \label{SHARP} f_{\sharp}(x_1, \xi_1) = (x_2, \xi_2), \;\; \mbox{with}\;\;
\left\{ \begin{array}{l} x_2 = f(x_1), \\ 
\xi_1 = df_{x_1}^* \xi_2 \in T^*_{x_1} X_1, \end{array} \right.  \end{equation}
where $$(df_{x_1})^* : T^*_{x_2} X_2 \to T_{x_1}^* X_1, $$
so that
$$f_{\sharp} |_{T^*_{x_1}} = (d f_{x_1})^{* -1}. $$
Then $f_{\sharp}$ is a symplectic diffeomorphism.

\subsection{\label{ISOMsect} Isometries: $F = T_g$}

We begin with  the simplest example where $F = T_g$
is translation by an isometry $T_g f(x) = f(g x)$ of a Riemannian manifold $(M,
ds^2)$. 
The canonical relation is the graph $\Gamma_g$ of the lift \eqref{SHARP} of
$T_g$ to $T^* M$.  Since $g$ is an isometry,
  $T_g$ commutes with $\Delta$ and we consider  thier  joint eigenfuntions with  $T_g \phi_j = e^{i \theta_j} \phi_j$ for some $e^{i \theta_j}
\in S^1$ (such as spherical harmonics $Y^{\ell}_m$ if $g$
is a rotation around the vertical axis).  The relevant space of
operators  $I^0(M \times M, \Gamma_g)$ is 
spanned by sums of operators $A T_g B$ with $A, B \in \Psi^0(M)$
and so it suffices to consider the functionals 
\begin{equation} \label{HATRHO} \hat{\rho}_j (AT_g B)  =  \langle A T_g B  \phi_j, \phi_j
\rangle = \langle A T_g B T_g^{-1} T_g  \phi_j, \phi_j
\rangle.  \end{equation}  Note that $A T_g B T_g^{-1} \in \Psi^0(M)$.    As one sees from the case $A = B = I$,
a  subsequence $\hat{\rho}_{j_k}$ can 
 only have a unique  weak* limit if
the associated eigenvalues $e^{i \theta_{j_k}}$ have a limit. To define matrix elements
with larger  subsequential limits, we re-normalize the functionals by 
\begin{equation} \label{METG} \rho_j (A T_g B): = \frac{ \langle A T_g B  \phi_j, \phi_j
\rangle}{\langle T_g \phi_j, \phi_j \rangle} = \langle A T_g B T_g^{-1} \phi_j, \phi_j
\rangle. \end{equation}
Here, we assume $\langle T_g \phi_j, \phi_j \rangle \not= 0$.  The reader may prefer
the $\hat{\rho}_j$; the methods and results apply equally to them and to \eqref{METG}.

We could regard the weak* limits as measures on $S^* M$ or on  $S \Gamma_g$, the graph of the lift of $T_g$
on $S^* M$, which  has the form $\{(\zeta,
g \cdot \zeta): \zeta \in S^* M\}$. To illustrate the lift, we  write the quantum
limit as \begin{equation}\label{FORMULA}  \rho_{\infty} (A T_g B)
= \int_{S^* M} a(\zeta) b(g \cdot \zeta) d \nu(\zeta, g \cdot
\zeta).
\end{equation}

\begin{prop} \label{ISOM} Let $g \in Isom(M, ds^2)$ and let $\nu_g$ be a weak* limit
measure for the functionals $\rho_j(F) = \frac{\langle F \phi_j, \phi_j \rangle}{T_g \phi_j,\phi_j \rangle}$
on $I^0(M \times M, C_g).$ Suppose that $T_g \phi_j = e^{- i \theta_j} \phi_j$.
Then under the identification $S^* M \to C_g, \zeta \to (\zeta, g \cdot \zeta)$,
$\nu = \pi^* \bar{\nu}$ is a signed measure of mass $\leq 1$ on $S^* M$ which is invariant under both
$G^t$ and $g$. \end{prop}

\begin{proof} We  have,
$$\int_{SC_g}  \sigma_{A T_g B } d\nu = \lim_{k \to \infty}  \frac{\langle A T_g B \phi_{j_k}, \phi_{j_k} \rangle}{\langle T_g \phi_{j_k}, \phi_{j_k} \rangle}. $$

Using $A = I$ and then $B =I $ we get  $\pi_* \nu = \bar{\nu},
\rho_* \nu = g_* \pi_* \nu = g_* \bar{\nu} - \nu. $ In the
notation (\ref{FORMULA}), this says, $$ \int_{S^* M} a(\zeta)  d
\nu(\zeta, g \cdot \zeta) = \int_{S^* M} a(\zeta) d
\omega(\zeta),$$ and $$ \int_{S^* M} b(\zeta)  d \omega(\zeta) =
\int_{S^* M}
 b(g \cdot \zeta) d \nu(\zeta, g \cdot \zeta) = \int_{S^* M}
 b(\zeta) d \nu(g^{-1} \zeta,  \zeta),$$
 which implies
 $$ \int_{S^* M} a(\zeta)  d
\nu(\zeta, g \cdot \zeta)  = \int_{S^* M} a(\zeta)  d
\nu(g^{-1}\zeta, \zeta), $$ i.e. $ d \nu(\zeta, g \cdot \zeta) = d
\nu(g^{-1} \zeta,  \cdot \zeta)$ or
 \begin{equation}\label{FORMULAa}   \int_{S^* M} a(\zeta) d \omega(\zeta) = \int_{S^*
M} a (g \zeta)  d \omega(\zeta).
\end{equation}

\end{proof}
\subsection{\label{SPHERE} Hecke operators on spheres}

The now consider the interesting example of self-adjoint Hecke operators on $S^n$, i.e. sums
\begin{equation} \label{T} T f(x) = \frac{1}{2d} \sum_{j = 1}^d (f(g_j x) + f(g_j^{-1} x) ), \;\;\; g_j \in SO(n + 1)  \end{equation}
of isometries on $S^n$. Note that $T$ is normalized so that $T 1 = 1$ and $||T|| = 1$. They are 
a helpful guide to Hecke correspondences on hyperbolic quotients
in the next section, and have a considerable literature of their own (see \cite{LPS} and subsequent articles).
The main result of this section,  Proposition \ref{HECKESN}, gives a new invariance principle for the quantum
limits of joint eigenfunctions of $T$ and $\Delta$. A key point is to view the limit measures as measures
on the canonical relation, which we may identify with $\bigcup_{j = 1}^k S^* S^n $,  rather than on $S^* S^n$.

A Hecke operator \eqref{T} is a discrete   Radon  transform $T = \rho_* \pi^*$
 corresponding to the trivial cover,

\begin{equation}\label{DIAGRAM} \begin{array}{ccccc} & & \bigcup_{j = 1}^{2d} S^n& &  \\ & & & & \\
& \pi \swarrow & & \searrow \rho & \\ & & & & \\
S^n & & & & S^n,
\end{array}
\end{equation} where $\pi(x, j) = x$ and $\rho (x, j) = g_j x. $
The canonical relation is the cotangent lift of the graph
$$\gcal_{T_{g_j}} = \{(z, g_j z): x \in S^n)\}$$ of the    isometric correspondence
$$C_T (x) = \{g_1 x, \dots, g_{2d} x\}. $$
 Thus we have a second
diagram

\begin{equation}\label{DIAGRAMb} \begin{array}{ccccc} & & \gcal_{T_g} \subset
S^n
\times S^n & &  \\ & & & & \\
& \pi \swarrow & & \searrow \rho & \\ & & & & \\
S^n  & & & & S^n .
\end{array}
\end{equation}
The graph of the
correspondence is an immersed submanifold of $S^n \times S^n$ with
self-intersections when the graphs of the individual isometries
$g_j$ intersect. This is clear since  $\iota(x, j) = \iota(x', j') \implies x =
x'$ and $g_j x = g_{j'} x$ or $g_{j'}^{-1} g_j x = x.$ 
The diagrams \eqref{DIAGRAM} and \eqref{DIAGRAMb} and are related as follows. 

\begin{lem}\label{Sparam}

The map
$$\iota: \bigcup_{j = 1}^{2d} S^n \to \gcal_{T_g}$$ defined by  $$\iota(\hat{x}) =
(\pi(\hat{x}),  \rho(\hat{x})), \;\;\; \mbox{or} \;\;\iota(x, j) =
(x, g_j x)
$$
is an immersion whose image is the graph. 

\end{lem}

The canonical relation $C_T \subset T^*(S^n \times S^n) $ of the Hecke operator $T$ is the graph of the sum of  cotangent lifts of the invidivual
isometries,  and the cotangent lift of the map in Proposition \ref{Sparam} gives a parametrization
\begin{equation} \label{iotaSn} \iota: \bigcup_{j = 1}^{2d} S^* S^n \to  C_T$$ defined by  $$\iota(\hat{x}) =
(\pi_{\sharp}(\hat{x}), \xi,  \rho(\hat{x}), \xi), \;\;\; \mbox{or} \;\;\iota_{\sharp}(x,  \xi, j) =
(x, g_j x, \xi, Dg_j^{* -1} \xi). 
\end{equation}
Thus, we may consider quantum limit measures as  measures on 
$\bigcup_{j = 1}^{2d} S^* S^n $, i.e. as a finite set $\{\half( \nu_j + \overline{\nu_j}\}_{j = 1}^{d}$ of real signed  measures on $S^* S^n$.

When $\langle T u_j, u_j \rangle \not= 0$, we can define the
Wigner functionals by the following normalized matrix elements (cf. \eqref{METG})
\begin{equation}\label{RHOHECKE}  \rho_{j, T} (A T B) = \frac{ \langle A T B
u_j, u_j \rangle}{\langle T u_j, u_j \rangle}, \;\; A, B \in
\Psi^0(S^n). \end{equation}
We denote by $\hat{\rho}_j$ the other normalization as in \eqref{HATRHO}.

We can also define the Wigner functionals for the individual
isometries $g_j$:
\begin{equation} \label{rhojgk} \rho_{j, g_k} (A T B) = \frac{ \langle A T_{g_k} B
u_j, u_j \rangle}{\langle T_{g_k} u_j, u_j \rangle}, \;\; A, B \in
\Psi^0(S^n). \end{equation}
Note that $$ \rho_{j, g_k^{-1}} (A T B) = \overline{\rho_{j, g_k}(B^* T A^*)}. $$

Exactly as in  Proposition \ref{ISOM}, the weak *  limits of the $\rho_{j, g_k} $ for each $k$ are linear functionals  of the form
\begin{equation} \label{NUK} \nu_k(A T B) = \int_{S^* S^n} a(\zeta) b(g_k \zeta) d\nu_k(\zeta, g_k \zeta),
\end{equation}  
although they are not necessarily invariant under $g_k$.
It follows that
$$\int_{\hat{\Gamma}_T} a(x) b(y) d\nu_T(x, y)) = \sum_k \int_{S^* S^n} a(\zeta) b(g_k \zeta) d\nu_k(\zeta, g_k \zeta).
$$
Lift $\{u_j\}$ to $\{\pi^* u_j\}$ on $\bigcup_{j = 1}^{2d} S^n$ and
let $\nu$ be a quantum limit measure of the sequence. If
$A, B$ are microsupported on the $k$th component, then the quantum
limit is $\nu_k$. That is,
$$\int_{\hat{\Gamma}_T} a(x) b(y) d\omega_T(x, y)) = \sum_k \int_{S^* S^n} a_k(\zeta) b_k(g_k \zeta) d\nu_k(\zeta, g_k \zeta).
$$
In Proposition \ref{ISOM}, the  measure $\nu_k$ is defined  on the graph of the lift of $T_{g_k}$ to $S^* M$; it may (and will)
be  identified 
with a measure on $S^* M$.  
Applying Proposition \ref{SA}, we get

\begin{prop} \label{HECKESN} Suppose that $\{u_{j_k}\}$ is a sequence of joint $\Delta - T$ eigenfunctions for which \eqref{rho} 
has a  unique weak * limit  $\mu$ on $S^* S^n$ and \eqref{METG} has a unique weak* limit $\nu$
on $\bigcup_{j = 1}^k S^* S^n$ as in Proposition \ref{Sparam}.
Then $\pi_* \nu = \rho_* \nu = \mu. $ Moreover,   each $\rho_{j,k, g, T}$ in \eqref{rhojgk} 
has a weak limit  $\nu_k$,  and
\begin{equation} \label{munuj} \mu =  \frac{1}{2d} \sum_{j = 1}^{2d} \nu_j =  \;\;  \frac{1}{2d} \sum_{j = 1}^{2d} (g_j)_* \nu_j =  \;\;  \frac{1}{2d} \sum_{j = 1}^{2d} (g_j^{-1})_* \nu_j.  \end{equation}
The measures $\nu_j $ are absolutely continuous with respect to $\mu$ and for each $g_j$, and also
$T_{g_j *} \nu_j$ is absolutely continuous with respect to $\mu$.
Similarly for all powers  $T^n$.
\end{prop}

\begin{rem} If we use $\hat{\rho}_j$ (cf. \eqref{HATRHO}) instead of  \eqref{RHOHECKE} then the statement should be:   the sequence of eigenvalues $\rho_{j_k}(T)$ of $T$-eigenvalues has a limit $\rho(T)$ and
\begin{equation} \label{munujhat} \mu =  \frac{1}{2d} \sum_{j = 1}^{2d} \nu_j =  \;\;\;  \frac{\rho(T)}{2d} \sum_{j = 1}^{2d} (g_j)_* \nu_j. \end{equation}
\end{rem}

\begin{proof}

First, the statement that $ \mu =  \frac{1}{2d} \sum_{j = 1}^{2d} \nu_j$ follows from the fact that the quantum limit $\nu$
is calculated on $\bigcup_{j = 1}^{2d} S^n$ using the pullbacks of $\{u_{j_k}\}$. If we test against an operator which is also
pulled back we get $2d$ times the quantum limit on the base.  The statement that
$\mu = \frac{1}{2d} \sum_{j = 1}^{2d} (g_j)_* \nu_j$ comes the fact
 that  $\{u_{j_k}\}$ is a sequence of $T$-eigenfunctions.Then the second equality is obvious by taking
limits of $\langle T A u_{j_k}, u_{j_k} \rangle$ for operators pulled back from the base.

Since $T_{g_k}$ is unitary,  
\begin{equation} \label{RIGHTINEQ} |\langle T_{g_k}A  \phi_j, \phi_j \rangle|^2 \leq \langle A^* A \phi_j, \phi_j \rangle   \end{equation}
and that implies $\nu_k << \mu$. Indeed,  \eqref{RIGHTINEQ} implies that along the relevant subsequence,
\begin{equation} \label{leq1} \lim |\langle T_{g_k} A  \phi_j, \phi_j \rangle|^2 = |\int_{S^* M} \sigma_A d\nu_{k} |^2 \leq \int_{S^* M } \sigma_A^2 d\mu.  \end{equation}
Also, if $A^* = A$,
 \begin{equation} \label{LEFTINEQ} |\langle T_{g_k} A  \phi_j, \phi_j \rangle|^2 = |\langle T_{g_k} A T_{g_k}^{-1} T_{g_k}  \phi_j, \phi_j \rangle|^2 \leq \langle (T_{g_k} A
T_{g_k}^{-1})^2 \phi_j, \phi_j \rangle   \end{equation}
and
\begin{equation} \label{ineq2}  \lim |\langle T_{g_k} A  \phi_j, \phi_j \rangle|^2 =
 |\int_{S^* M} \sigma_A d\nu_k |^2 \leq \int_{S^* M } (T_{g_k}\sigma_A)^2 d\mu, \end{equation} 
and therefore $d\nu_k << g_{k *} ^{-1} d\mu$.  Also
 \begin{equation} \label{RIGHTINEQb} |\langle A  T_{g_k}   \phi_j, \phi_j \rangle|^2 = |\langle T_{g_k}  T_{g_k}^{-1}  A T_{g_k}  \phi_j, \phi_j \rangle|^2 \leq \langle (T_{g_k}^{-1} A
T_{g_k})^2 \phi_j, \phi_j \rangle   \end{equation}
and
\begin{equation} \label{ineq3}  \lim |\langle T_{g_k} ^{-1}A  \phi_j, \phi_j \rangle|^2 =
 |\int_{S^* M} \sigma_A d\nu_k |^2 \leq \int_{S^* M } (T_{g_k}^{-1}\sigma_A)^2 d\mu, \end{equation} 
and therefore $d\nu_k << g_{k *} d\mu$.

Here we are using  \eqref{LEFTRIGHT}-\eqref{LRLIMIT}  termwise on each term and the fact that $\phi_j$ is a $T$-eigenfunction.  That is, 
$\nu_j$ is essentially the left limit and $g_{j *} \nu_j$ is the right limit. Inequality \eqref{RIGHTINEQ}
is the inequality for the right functional and \eqref{LEFTINEQ} is the inequality for the left functional.  


\end{proof}

\begin{rem} If instead of $\rho_j$ we use $\hat{\rho}_j$, 
then in order for the weak limit  to exist, we need afortiori that the sequence  $\langle T u_{j_k}, u_{j_k} \rangle$ tends to a limit,
i.e. that the associated sequence of eigenvalues tends to a limit. The rest proceeds as above.
\end{rem}

\begin{cor}  \label{INVCOR} In the notation of Proposition \ref{HECKESN}, 
let $\Lambda_j$ be the singular support of $\nu_j$ and let   $\Lambda$ be the singular support of $\mu$. Then $\Lambda = \bigcup_j \Lambda_j = \bigcup_j g_j(\Lambda_j) . $  \end{cor}

\begin{proof}  Let $\Lambda_j$ be the singular support of $\nu_j$. In view of  \eqref{munuj}, 
$\Lambda \subset \bigcup_{j = 1}^d \Lambda_j$. Strict inclusion is 
apriori possible since cancellations may occur among the terms.
On the other hand, since all $\nu_j << \mu$,
 $\Lambda_j \subset \Lambda$ for all $j$ and therefore
$\bigcup_{j = 1}^{2d} \Lambda_j \subset \Lambda. $
It follows that
\begin{equation} \Lambda = \bigcup_{j = 1}^{2d} \Lambda_j. \end{equation}

It also follows from \eqref{munuj} that  $\Lambda \subset \bigcup_{j = 1}^{2d} g_j \Lambda_j$, since 
 the singular support of  $g_{j*} \nu_j$ is $g_j \Lambda_j$. But also by \eqref{ineq2}, $g_{j*} \nu_j <<  \mu$,
so $g_{j*} \nu_j $ has  singular support in  $\Lambda$ for all j, i.e. $g_j \Lambda_j \subset \Lambda$ and therefore
\begin{equation} \Lambda = \bigcup_{j = 1}^{2d} g_j \Lambda_j. \end{equation}




\end{proof}

\begin{rem} \label{MORE}  If $T \phi_j = \rho_j(T) \phi_j$ then also $T^{\ell} \phi_j = \rho_j(T)^{\ell} \phi_j$. 
The Proposition and Corollary apply equally to all powers of $T$. \end{rem}

For generic $T$, we can obtain some simple restrictions on weak* limits from Corollary \ref{INVCOR} and Remark \ref{MORE}.
For simplicity, we assume that $n = 2$, so that each maximal abelian subgroup is one dimensional. We assume that
each $g_j$ is a topological generator in a maximal abelian subgroup, i.e. that it is a rotation with an irrational angle. We also
assume that the $g_j$ are pairwise independent in the sense that no two generate the same circle of rotations (unless
they are inverses). 

As a sample application, we show that $\Lambda$ cannot be a single closed geodesic $\gamma$, i.e. $d\mu$ cannot
be a constant multiple of the invariant probability measure $\mu_{\gamma}$ on $\gamma$. 
Otherwise, we would have $\nu_j = c_j \mu_{\gamma}$ with $|c_j |\leq 1$  for all $j$. Certainly  $c_{j_0} \not= 0$ for some $j_0$.
Corollary \ref{INVCOR}  then  forces $g_{j_0} \gamma = \gamma$, i.e. $g_{j_0}$ is a generator of $\gamma$. But by the assumption that the $g_j$ are independent,  this forces  the other $c_j = 0$
(i.e. except for $g_{j_0}, g_{j_0}^{-1}$).  By Proposition \ref{HECKESN} we would then have $\mu_{\gamma} = \frac{1}{2d}  (c_{j_0} + \overline{c_{j_0} })  \mu_{\gamma}
= \frac{\rho(T)}{2 d} (c_{j_0} + \overline{c_{j_0} })  \mu_{\gamma}. $ This leads to 
several contradictions. First, $|c_{j_0} | \leq 1$ and $d \geq 1$, so it is impossible
that $1 = \frac{1}{2d}  (c_{j_0} + \overline{c_{j_0} })$.  The equation also   forces $\rho(T) = 1$.
This is  a contradiction when $T$ has a spectral gap \cite{LPS}.  Further contradictions arise if we consider powers of $T$, since almost all terms in the weak
* limit sums must vanish for similar reasons.

\subsection{Hecke operators on hyperbolic quotients}

We now generalize Proposition \ref{HECKESN} to arithmetic hyperbolic quotients. The arguments
are essentially the same, with $\H$ replacing $S^n$ everywhere. The one difference is that we have 
an additional discrete group $\Gamma$ operating which commutes with the Hecke operator so
that it acts on the quotient.  However, we only need to use it to ensure that the relevant cover of $\H$ is
finite sheeted. 

Let $\Gamma$ is a co-compact (or cofinite) discrete subgroup of $G
= PSL(2, \R)$ and let  $\X$ be the corresponding compact (or
finite area) hyperbolic surface. An element $g \in G, g \notin
\Gamma$ is said to be in the commensurator $Comm(\Gamma)$ if
$$\Gamma'(g) : = \Gamma \cap g^{-1} \Gamma g$$
is of finite index in $\Gamma$ and $g^{-1} \Gamma g$. More
precisely,
$$\Gamma = \bigcup_{j = 1}^d \Gamma'(g) \gamma_j, \;\;
(\mbox{disjoint}), $$ or equivalently
$$\Gamma g \Gamma = \bigcup_{j = 1}^d \Gamma \alpha_j, \;\; \mbox{where}\;\;\alpha_j = g \gamma_j. $$
It is also possible to choose $\alpha_j$ so that $\Gamma g \Gamma
=   \bigcup_{j = 1}^d \alpha_j \Gamma. $ We refer to \cite{Sh} for background, or \cite{Z1} for 
 notation of this section.

 Similar to \eqref{DIAGRAM}, we  then have
a diagram of finite (non-Galois) covers:
\begin{equation}\label{DIAGRAM1} \begin{array}{ccccc} & & \Gamma'(g) \backslash \H & &  \\ & & & & \\
& \pi \swarrow & & \searrow \rho & \\ & & & & \\
\Gamma \backslash \H & & \iff & & g^{-1} \Gamma g \backslash \H .
\end{array}
\end{equation}
Here,
$$\pi(\Gamma'(g) z) = \Gamma z, \;\; \rho(\Gamma'(g) z) = g \Gamma
g^{-1} (g \gamma_j z) = g \Gamma z,  $$ where in the definition of
$\rho$ any of the $\gamma_j$ could be used.  The horizontal map is
$z \to g^{-1} z$.
The Radon transform $\rho_* \pi^*$ of the diagram defines a Hecke
operator $T_g: L^2 (\X) \to L^2(\X)$,
$$T_g f(x) = \frac{1}{d} \sum_{j = 1}^d f(g \gamma_j x). $$ Then
$T_{g^{-1}} u(z) = u(g^{-1} z)$
takes $\Gamma$-invariant functions to $g^{-1} \Gamma g$-invariant
functions. 
Hecke operators are self-adjoint and commute with the hyperbolic Laplacian and we can
consider joint eigenfunctions of $\Delta$ and $T_g$,
$$T_g u_j = \rho_j(T_g) u_j. $$
The main difference to the case of $S^n$ is that the covering  \eqref{DIAGRAM1} is not trivial, i.e.
not a disjoint union of $d$ copies of the base. However, below we uniformize so that it  does become
trivial.

The Hecke operator is a kind of averaging operator over orbits of
the Hecke correspondence, which is the multi-valued holomorphic
map
$C_g (z) = \{\alpha_1 z, \dots, \alpha_d z\}. $
Its graph
$\gcal_g = \{(z, \alpha_j z): z \in \X)\}$
is an algebraic curve in  the product $\X \times g^{-1}
\Gamma g \backslash \H. $ Similar to \eqref{DIAGRAMb} we   have a second diagram

\begin{equation}\label{DIAGRAMa} \begin{array}{ccccc} & & \gcal_g \subset \X
\times g^{-1} \Gamma g \backslash \H & &  \\ & & & & \\
& \pi_1 \swarrow & & \searrow \rho_1 & \\ & & & & \\
\Gamma \backslash \H & & \iff & & g^{-1} \Gamma g \backslash \H .
\end{array}
\end{equation}
Here, $\pi_1, \rho_1 $ are the natural projections.
The following is a kind of analogue of Lemma \ref{Sparam}.

\begin{lem}

The map
$$\iota: \Gamma'(g)
\backslash \H \to \Gamma_g$$ defined by  $$\iota(z) = (\pi(z),
\rho(z)), $$ is a local diffeomorphic parametrization, and $\pi_1
\iota = \pi, \rho_1 \iota = \rho. $

\end{lem}

\begin{proof} It is obvious that $\iota$ is well-defined, takes its values in $\Gamma_g$  and
intertwines the projections. Each of the maps $\pi, \rho$ is
itself a local diffeomorphism and therefore $\iota$ also is. The
`fiber' over $(z, \alpha_j z)$ is multiple if and only if there
exist $\alpha_j, \alpha_k, j \not= k$ such that $\alpha_j z =
\alpha_k z$ and this can only happen for a finite set of $z$.

%
%

\end{proof}

We now uniformize and  consider Hecke operators on $\H$ as $\Gamma$-periodic
versions of the Hecke operators on spheres. 
We then obtain a picture similar to that
of $S^n$, where
\begin{equation}\label{DIAGRAMc} \begin{array}{ccccc} & & \bigcup_{j = 1}^d \H \times \{j\}& &  \\ & & & & \\
& \pi \swarrow & & \searrow \rho & \\ & & & & \\
\H & & & & \H,
\end{array}
\end{equation} where $\pi(x, j) = x$ and $\rho (x, j) = g \gamma_j x. $
The Radon transform $\rho_* \pi^*$ of the diagram defines the
Hecke operator $T_g: L^2 (\H) \to L^2(\H)$.
Exactly as  in \eqref{iotaSn} in th case of  $S^n$, 
the map
\begin{equation} \label{HYPANAG} \iota: \bigcup_{j = 1}^d \H \times \{j\} \to \gcal_{T_g} \subset \H \times \H$$ defined by  $$\iota(\hat{x}) =
(\pi(\hat{x}),  \rho(\hat{x})), \;\;\; \mbox{or} \;\;\iota(x, j) =
(x, g \gamma_j x)
 \end{equation}
is an immersion whose image is the graph $\Gamma_T$.
The self-intersection points occur when $\iota(x, j) = \iota(x', j') \implies x =
x'$ and $g_j x = g_{j'} x$ or $g_{j'}^{-1} g_j x = x.$ 
If we take the quotient by $\Gamma$ of diagram (\ref{DIAGRAMb}) we 
get  diagram (\ref{DIAGRAM1}). But it is preferable to regard all of the functionals as defined on pseudo-differential
operators on $\H$ with compact spatial support. Taking the quotient by $\Gamma$ amounts to cutting off 
$\Gamma$-invariant pseudo-differential operators to fundamental domains, but nothing is lost (and some generality
is gained) by using non $\Gamma$-invariant pseudo-differential operators.

\subsection{Quantum limits on $\H$}

On $\H$ we can define the Wigner functionals  for the individual
isometries $g_j$:
\begin{equation} \rho_{j, g_k} (A T B) = \frac{ \langle A T_{g_k} B
u_j, u_j \rangle}{\langle T u_j, u_j \rangle}, \;\; A, B \in
\Psi_c^0(\H). \end{equation}

By Proposition \ref{ISOM}, the limits have the form
$$\int_{S^* \H} a(\zeta) b(g_k \zeta) d\nu_k(\zeta, g_k \zeta).
$$
As with $S^n$ it follows that
$$\int_{\hat{\Gamma}_T} a(x) b(y) d\nu_T(x, y)) = \sum_k \int_{S^* \H} a(\zeta) b(g_k \zeta) d\nu_k(\zeta, g_k \zeta).
$$
Here, $d\nu_k$ is a measure on the $k$th copy of $\H$. The  Hecke
limit measure is the sum
$$\nu = \sum_j \nu_{g_j} \;\;\; \mbox{on} \;\; \bigcup_j S^* \H \times\{j\},
$$
with $\nu_{g_j}$ living on the jth copy $S^* \H \times \{j\}$.

When $\langle T_p u_j, u_j \rangle \not= 0$, we can define the
Wigner functionals by the following normalized matrix elements:
\begin{equation} \rho_{j, g} (A T_p B) = \frac{ \langle A T_p B
u_j, u_j \rangle}{\langle T_p u_j, u_j \rangle}, \;\; A, B \in
\Psi_c^0(\H). \end{equation}

Completely analogously to Proposition \ref{HECKESN}, we have:

\begin{prop} \label{HECKEHYP} Suppose that $\{u_{j_k}\}$ is a sequence of joint $\Delta - T$ eigenfunctions for which \eqref{rho} 
has a  unique weak * limit  $\mu$ on $S^* \H$ and \eqref{METG} has a unique weak* limit $\nu$
on $\bigcup_{j = 1}^{d} S^* \H$.
Then $\pi_* \nu = \rho_* \nu = \omega. $ Moreover,   each $\rho_{j,k, g, T}$ in \eqref{rhojgk} 
has a weak limit  $\nu_k$  and
\begin{equation} \label{munujhyp} \mu = \frac{1}{ d} \sum_{j = 1}^{d} \nu_j =  \;\;\;  \frac{1}{ d} \sum_{j = 1}^{2d} (g_j)_* \nu_j. \end{equation}
The measures $\nu_j $ are absolutely continuous with respect to $\mu$ and for each $g_j$, and also
$T_{g_j *} \nu_j$ is absolutely continuous with respect to $\mu$.
Similarly for all powers  $T^n$. Moreover, $\mu$ is $\Gamma$-invariant. 
\end{prop}

The proof is essentially the same as for Proposition \ref{HECKESN} and is omitted. The fact that $\H$ is of infinite volume
and the pseudo-differential operators are spatially compactly supported does not change the proof in any significant way. The only
new statement is that $\mu$ is $\Gamma$-invariant, which is obvious from the fact that $T$ commutes with $\Gamma$.

\begin{cor}  \label{INVCORH} In the notation of Proposition \ref{HECKESN}, 
let $\Lambda_j$ be the singular support of $\nu_j$ and let   $\Lambda$ be the singular support of $\mu$. Then $\Lambda = \bigcup_j \Lambda_j = \bigcup_j g_j(\Lambda_j) . $  \end{cor}

Again the proof is the same as for Corollary \ref{INVCOR}.

\subsection{Quantum limits on $S^* \Gamma'(g) \backslash \H$}

 Instead of viewing $\nu$ as a measure on $\bigcup_j S^* \H \times \{j\}$ as in Proposition \ref{HECKEHYP}, we may
take the quotient by $\Gamma$ and view it as a measure on $S^* (\Gamma'(g) \backslash \H) = \Gamma'(g) \backslash G. $
Note that the quotient by $\Gamma$  glues together the $d$ disjoint copies of $\H$ and then takes
the quotient of the resulting $\H$  by $\Gamma'(g)$. To explain the $\Gamma$-invariance properties we prove

\begin{lem} $\sum_j \nu_j$ is a $\Gamma$-invariant measure on $S^* \H \simeq G$.
 \end{lem}

 \begin{proof}

Let $A, B$ be compactly supported pseudo-differential operators on
$\H$. Since $T_g L_{\gamma} = L_{\gamma} T_g$,
$$\begin{array}{lll} \langle A T_g B u_j, u_j \rangle & = &\langle A T_g B L_{\gamma} u_j, L_{\gamma} u_j \rangle \\ && \\
& = & \langle L_{\gamma}^* A T_g B L_{\gamma} u_j, u_j \rangle \\ && \\
& = & \langle L_{\gamma}^* A L_{\gamma} T_{g}
L_{\gamma}^* B L_{\gamma} u_j, u_j \rangle. \end{array}$$

Taking the limit gives $\gamma_* \nu = \nu.$


 \end{proof}



 \end{document}